\documentclass{article}
\usepackage{amssymb}
\usepackage{amsmath}
\usepackage{color}
\usepackage{amsthm}
\usepackage{comment}

\newtheorem{theorem}{Theorem}
\newtheorem{lemma}[theorem]{Lemma}
\newtheorem{definition}[theorem]{Definition}
\newtheorem{example}[theorem]{Example}
\newtheorem{corollary}[theorem]{Corollary}
\newtheorem{proposition}[theorem]{Proposition}
\title{Patterns of ideals of numerical semigroups}
\author{Klara Stokes\\
School of Engineering Sciences\\
University of Sk\"ovde, Box 408\\
54128 Sk\"ovde Sweden\\
klara.stokes@his.se}
\date{}
\begin{document}
\maketitle
\begin{abstract}
This article introduces patterns of ideals of numerical semigroups, thereby unifying previous definitions of patterns of numerical semigroups. Several results of general interest are proved. More precisely, this article presents results on the structure of the image of patterns of ideals, and also on the structure of the sets of patterns admitted by a given ideal. 
\end{abstract}
\section{Introduction}

%Definition and properties of numerical semigroups in general

A \emph{numerical semigroup} $S$ is a subset of the non-negative integers (denoted by $\mathbb{Z}_+$) that contains zero, is closed under addition and has finite complement in $\mathbb{Z}_+$. The set of non-zero elements in $S$ is denoted by $M(S)$. 

%Gaps genus and multiplicity
Elements in the complement $\mathbb{Z}_+\setminus S$ are called \emph{gaps}, and the number of gaps is the \emph{genus} of $S$. 
The smallest element of $M(S)$ is the \emph{multiplicity} of $S$ and it is denoted by $m(S)$. 
%Frobenius and PseudoFrobenius
The largest integer not in $S$ is the \emph{Frobenius element} and is denoted by $F(S)$. The number $F(S)+1$ is called the \emph{conductor} of $S$ and is denoted by $c(S)$. 
An integer $x\not\in S$ is \emph{pseudo-Frobenius} if $x+s\in S$ for all $s\in M(S)$. 
The set of pseudo-Frobenius integers is denoted by $PF(S)$. Note that $F(S)\in PF(S)$ and that $F(S)$ is the maximum of the elements in $PF(S)$.  

%Generators
It can be proved that, given a numerical semigroup $S$, there exists a unique minimal set of elements $B\subset M(S)$ such that any element in $S$ can be expressed as a linear combination of elements from $B$.  The elements in $B$ are called \emph{minimal generators} of $S$ and they are exactly the elements of $M(S)$ that can not be obtained as the sum of two elements of $M(S)$. The cardinality of $B$ is always finite. More precisely it is always less or equal to the multiplicity of $S$. 
A numerical semigroup has \emph{ maximal embedding dimension} if the number of minimal generators equals the multiplicity. 

%Ideals
A \emph{relative ideal} of a numerical semigroup $S$ is a set $H\subseteq \mathbb{Z}$ satisfying $H+S\subseteq H$ and $H+d\subseteq S$ for some $d\in S$. 
A relative ideal contained in $S$ is an \emph{ideal} of $S$. 
An ideal is \emph{proper} if it is distinct from $S$. 
The set of proper ideals of $S$ has a maximal element with respect to inclusion. 
This ideal is called the \emph{maximal ideal} of $S$, and equals $M(S)$, the set of non-zero elements of $S$.  
The \emph{dual} of a relative ideal $H$ is the relative ideal $H^*=(S-H)=\{z\in \mathbb{Z}:z+H\subseteq S\}$. 

%Definition and properties of patterns
A \emph{pattern} admitted by an ideal $I$ of a numerical semigroup $S$ is a multivariate polynomial function $p(X_1,\dots,X_n)$ which returns an element $p(s_1,\dots,s_n)\in S$ when evaluated on any non-increasing sequence $(s_1,\dots,s_n)$ of elements in $I$. We say that the ideal $I$ \emph{admits} the pattern. 
If $I=S$, then we say that the numerical semigroup $S$ admits the pattern. 
Note that a pattern admitted by an ideal $I$ of a numerical semigroup $S$ is also admitted by any ideal $J\subseteq I$. 

We will identify the pattern with its polynomial, and, for example, say that the pattern is linear and homogeneous, when the pattern polynomial is linear and homogeneous. 
The length of a pattern is the number of indeterminates and its degree is the degree of the pattern polynomial. 
One pattern $p$ \emph{induces} another pattern $q$ if any ideal of a numerical semigroup that admits $p$ also admits $q$. 
Two patterns are equivalent if they induce each other. 
If any ideal satisfying a given condition $c$ that admits $p$ also admits $q$, then we say that $p$ induces $q$ under the condition $c$. 
Two patterns are equivalent under the condition $c$ if they induce each other under the condition $c$. 

%Homogeneous linear patterns were introduced by Bras-Amor\'os and Garc\'ia-S\'anchez in \cite{MariaPedro}, and non-homogeneous linear patterns were treated in \cite{MariaPedroAlbert}. 

%Examples of homogeneous patterns

%We end this introduction with different examples of patterns admitted by numerical semigroups. 
Homogeneous linear patterns admitted by numerical semigroups were introduced by Bras-Amor\'os and Garc\'ia-S\'anchez in \cite{MariaPedro}. 
The patterns that they considered were all defined by homogeneous linear multivariate polynomials with the whole numerical semigroup as domain. 
Examples of homogeneous patterns are the homogeneous linear patterns with positive coefficients.  
It is easy to see that these patterns are admitted by any numerical semigroup. 
Arf numerical semigroup are characterized by admitting the homogeneous linear ``Arf pattern'' $X_1+X_2-X_3$. 
Homogeneous linear patterns of the form $X_1+\cdots+X_{k}-X_{k+1}$ generalise the Arf pattern and are called \emph{subtraction patterns}~\cite{MariaPedro}. 
%The positive integer $k$ is called the degree of the subtraction pattern. In~\cite{MariaPedro} it was proved that a numerical semigroup $S$  with multiplicity $m(S)$ and conductor $c(S)$ always admits a subtraction pattern of degree $\lceil\frac{c}{m}\rceil+1$. %klara tak. 
%It can be proved that a numerical semigroup that admits a subtraction pattern of degree $k$ also admits a subtraction pattern of degree $k+n$ for all $n\geq 0$. 
%The smallest number $k$ such that $S$ admits a subtraction pattern of degree $k$ is called the subtraction degree of $S$. 
%The only numerical semigroup of subtraction degree 1 is $\mathbb{Z}_+$. The numerical semigroups of subtraction degree 2 are exactly the Arf numerical semigroups that are different from $\mathbb{Z}_+$.  

%Examples of non-homogeneous patterns
The definition of pattern from $S$ to $S$ does not allow for non-homogeneous patterns with constant term outside $S$. 
To overcome this problem, when the non-homogeneous patterns were introduced in \cite{MariaPedroAlbert}, it was with $M(S)$ as domain. 
Note that with this definition  $X+a$ with $a\in PF(S)$ is a non-homogeneous linear pattern admitted by $S$.  
Admitting the non-homogeneous linear pattern $X_1+X_2-m(S)$ is equivalent to the property of maximal embedding dimension. 
Since this pattern is induced by the pattern $X_1+X_2-X_3$, this implies that a numerical semigroup that is Arf is always of maximal embedding dimension. 

Further examples of non-homogeneous linear patterns can be found in the numerical semigroups associated to the existence of combinatorial configurations (see \cite{BrasStokes}). It was proved in \cite{StokesBras1,StokesBras2} that such numerical semigroups admit the patterns $X_1+X_2-n$ for $n\in \{1,\dots,\gcd(r,k)\}$ and $X_1+\dots+X_{rk/\gcd(r,k)}+1$, where $r$ and $k$ are positive integers that depend on the parameters of the combinatorial configuration. This example motivates the study of a set of patterns that are admitted simultaneously by the same numerical semigroup. 

Another example of a non-homogeneous linear pattern is  $qX_1-qm(S)$, which is admitted by a Weierstrass semigroup $S$ of multiplicity $m(S)$  of a rational place of a function field over a finite field of cardinality $q$, for which the Geil-Matsumoto bound and the Lewittes' bound coincide~\cite{BrasVico}.
Similarly, the pattern $(q-1)X_1-(q-1)m(S)$ is admitted if and only if the Beelen-Ruano's bound equals $1+(q-1)m$~\cite{MariaPedroAlbert}.  %{\color{red}(Check all this). }

Patterns can be used to explore the properties of the numerical semigroup admitting them. For example, the calculations of the formulae for the notable elements of Mersenne numerical semigroups in ~\cite{Branco} rely on the fact that all Mersenne numerical semigroups generated by a consecutive sequence of Mersenne numbers admit the non-homogeneous pattern $2X_1+1$. 
Similarly, the non-homogeneous patterns admitted by numerical semigroups associated to the existence of combinatorial configurations were used to improve the bounds on the conductor of these numerical semigroups. 

In this article we study patterns of ideals of numerical semigroups. 

Section~\ref{sec:image} contains basic results about the properties of the image of patterns. For example, it is proved that if the greatest common divisor of the coefficients of the pattern $p$ is one and $I$ is an ideal of a numerical semigroup $S$, then the image $p(I)$ of a pattern is always an ideal of a numerical semigroup. Also, sufficient conditions are given for when $p(I)\subseteq S$. 

Section~\ref{sec:calculate} presents an upper bound of the smallest element $c$ in $p(I)$ such that all integers larger than $c$ belong to $p(I)$, under the condition that the greatest common divisor of the coefficients of $p$ is one. By dividing $p$ by the greatest common divisor of its coefficients, this result makes it possible to calculate $p(I)$ for any admissible pattern $p$. 

Section~\ref{sec:patternsofideals} introduces the concepts endopattern and surjective pattern of an ideal, and gives some sufficient and necessary conditions on patterns to have these properties. 

In Section~\ref{sec:closure}, we generalize the notion of closure of a numerical semigroup with respect to a homogeneous linear pattern to the closure of an ideal of a numerical semigroup with respect to a non-homogeneous linear pattern. We also prove a necessary condition for when the closure of an ideal with respect to a non-homogeneous pattern can be calculated by repeatedly applying the pattern. 

In Section~\ref{sec:structures} we prove that the set of patterns admitted by an ideal of a numerical semigroup has the structure of a semigroup, semiring or semiring algebra, depending on if the length and the degree of the patterns is fixed. 

Section~\ref{sec:pseudofrobenius} introduces a generalization of pseudo-Frobenius as a useful tool in the analysis of the structures defined in Section~\ref{sec:structures}.

Section~\ref{sec:image2} introduces infinite chains of ideals of numerical semigroups where the subsequent ideal $I_{i}$ is the image of the preceding ideal $I_{i-1}$ under a pattern $p$ which is admitted by the first pattern in the chain, and hence by them all.

Finally, Section~\ref{sec:composition} defines polynomial composition of patterns, hence providing yet another operation that creates a pattern admitted by an ideal from several patterns admitted by that ideal.

\section{The image of a pattern}
\label{sec:image}
A pattern $p$ admitted by an ideal $I$ of a numerical semigroup $S$ returns elements in $S$ when evaluated over the non-increasing sequences of elements of $I$. 
We will now study the image $p(I)$ of $I$ under $p$. 

We will need the following well-known result. 
\begin{lemma}
\label{lemma_fincompl}
Let $A\subseteq \mathbb{Z}_+$ be closed under addition. 
Then $A$ does not have finite complement in $\mathbb{Z}_+$ if and only if $A\subseteq u\mathbb{Z}$ for some $u>1$. 
\end{lemma}
\begin{proof} 
If $A$ does not have finite complement, then $A$ does not contain $x,y$ such that $\gcd(x,y)=1$, since otherwise the maximal ideal of the numerical semigroup $\langle x,y\rangle$, which has finite complement,  would be contained in $ A$. 
Therefore $\gcd(A)=u$ for some $u>1$ so that $A\subseteq u\mathbb{Z}$. 

If $u\mathbb{Z}$ is an ideal of $\mathbb{Z}$ with $u>1$ and $A\subseteq u\mathbb{Z}$, then clearly $|\mathbb{Z}_+\setminus A|$ is infinite. 
\end{proof}
%Although the next result is proven here for images of the form $p(S)\subseteq S$ for $p$ a pattern of $S$, it is easily adapted to the situation $p(M(S))\subseteq M(S)$, by instead considering the set $p(M(S))\cup \{0\}$.  

\begin{lemma}
\label{lem:conductor}
If $I$ is an ideal of some numerical semigroup $S$, then there is a $c\in I$ such that $z\in I$ for all $z\in \mathbb{Z}$ with $z\geq c$. 
\end{lemma}
\begin{proof}
If $I$ is an ideal of some numerical semigroup $S$, then $I+S\subseteq I$, implying that $|\mathbb{Z}_+\setminus I|<\infty$. Therefore there is a $c\in I$ such that $z\in I$ for all $z\in \mathbb{Z}$ with $z\geq c$. 
\end{proof}
If, in Lemma \ref{lem:conductor}, $I=S$, then the integer $c$ is the conductor of $S$. If $I$ is a proper ideal, then we call $c$ the \emph{maximum of the small elements of $I$. }

\begin{theorem}
\label{thmhomideals}
Let $p(X_1,\dots,X_n)=a_1X_1+\cdots+a_nX_n$ be a homogeneous linear pattern admitted by $\mathbb{Z}_+$ and let $I$ be an ideal of a numerical semigroup $S$. %by an ideal $I$ of a numerical semigroup $S$. 
Then $p(I)$ is an ideal of some numerical semigroup if and only if $\gcd(a_1,\dots,a_n)=1$. 
\end{theorem}
\begin{proof}
Assume that $\gcd(a_1,\dots,a_n)=1$ and let $c$ be the maximum of the small elements of $I$ (see Lemma \ref{lem:conductor}).  
Let $u>1$ and $s\in I\cap u\mathbb{Z}$ with $s\geq c$. Then $s+1,s+u\in I$, with $s+1\not\in u\mathbb{Z}$, $s+u\in u\mathbb{Z}$ and $s+u>s+1>s$. 
Since $\gcd(a_1,\dots,a_n)=1$ there is an $i\in [1,n]$ such that $a_i$ is not a multiple of $u$. 
Therefore  $\sum_{j=1}^{i-1}a_j(s+u)+a_i(s+1)+\sum_{j=i+1}^{n}a_js \in p(I)\setminus u\mathbb{Z}$. 
Lemma~\ref{lemma_fincompl} implies that $p(I)$ has finite complement in $\mathbb{Z}_+$. 

Note that if  $x_1,\dots,x_n$ and $y_1,\dots,y_n$ are non-increasing sequences of $I$, then so is $x_1+y_1,\dots,x_n+y_n$. 
Since the pattern $p$ is linear and homogeneous we have $p(x_1,\dots,x_n)+p(y_1,\dots,y_n)=p(x_1+y_1,\dots,x_n+y_n)\in p(I)$ for all non-increasing sequences $x_1,\dots,x_n\in I$ and $y_1,\dots,y_n\in I$, that is, $a+b\in p(I)$ for all $a,b\in p(I)$.  Hence $p(I)$ is closed under addition. (That linearity of $p$ implies that $p(I)$ is closed under addition was first noted in \cite{MariaPedro}.)

Together, the above imply that if $0\in p(I)$, then $p(I)$ is a numerical semigroup, and if $0\not\in p(I)$, then $p(I)$ is the maximal ideal of the numerical semigroup $p(I)\cup \{0\}$. In any case, $p(I)$ is an ideal of a numerical semigroup. 

Now assume that $\gcd(a_1,\dots,a_n)=u>1$. Then clearly $p(I)\subseteq u\mathbb{Z}$, so that $p(I)$ does not have finite complement in $\mathbb{Z}_+$ and can not be the ideal of any numerical semigroup.
\end{proof}
Note that in Theorem~\ref{thmhomideals}, either $p(I)$ is a numerical semigroup, or $p(I)$ is the maximal ideal of the numerical semigroup $p(I)\cup \{0\}$, depending on whether $0\in p(I)$ or not. 
When $I$ is a proper ideal, then $0\in p(I)$ exactly when $\sum_{i=1}^na_i=0$ (see Proposition~\ref{propendM}). %check 
When $I=S$, then Theorem~\ref{thmhomideals} implies the following result.
\begin{corollary}
\label{lemmahomsem}
Let $p(X_1,\dots,X_n)=a_1X_1+\cdots+a_nX_n$ be a homogeneous linear pattern admitted by $\mathbb{Z}_+$ and let $S$ be a numerical semigroup.  % the numerical semigroup $S$. 
Then $p(S)$ is a numerical semigroup if and only if $\gcd(a_1,\dots,a_n)=1$.
\end{corollary}
\begin{proof}
Apply Theorem~\ref{thmhomideals} with $I=S$ and note that $p(0,\dots,0)=0\in p(S)$. 
\end{proof}
Clearly the numerical semigroup $p(S)$ is contained in the original numerical semigroup $S$ if and only if $p$ is admitted by $S$.  

%\begin{corollary}
%Let $I$ be an ideal of a numerical semigroup and let $c(I)$ be the smallest element in $I$ such that $z\in I$ for all integers $z\geq c$. 
%Let $p(X_1,\dots,X_n)=\sum_{i=1}^na_iX_i$ be a homogeneous linear pattern with $\gcd(a_1,\dots,a_n)=d$ and let $b_1,\dots,b_n$ be the (unique) integers such that $a_1b_1+\cdots+a_nb_n=d$. 
%Then 

Following \cite{MariaPedro}, a linear homogeneous pattern $p(X_1,\dots,X_n)=\sum_{i=1}^na_iX_i$ is \emph{premonic} if $\sum_{i=1}^{n'}a_i=1$ for some $n'\leq n$. We say that a linear non-homogeneous pattern $p(X_1,\dots,X_n)=\sum_{i=1}^na_iX_i+a_0$ is premonic if $p(X_1,\dots,X_n)-a_0$ is premonic. If $a_1=1$ then $p$ is \emph{monic} and so all monic patterns are premonic.  

\begin{lemma}
\label{lemmapremonic}
If $p$ is a premonic linear homogeneous pattern admitted by a numerical semigroup $S$, then $p(S)=S$.  
\end{lemma}
\begin{proof}
If $p$ is a linear homogeneous pattern admitted by $S$ and $p$ is premonic, then $\sum_{i=1}^{n'}a_i s + \sum_{j=n'+1}^na_j 0 =s$ for all $s\in S$, so that $S\subseteq p(S)$. 
Clearly $p(S)\subseteq S$, so that $p(S)=S$. 
\end{proof}

%By Lemma~\ref{lemmapremonic}, a linear premonic homogeneous endopattern is surjective. 

Moreover, the image of a premonic linear pattern, homogeneous or not, admitted by an ideal $I$ of a numerical semigroup $S$, is an ideal of $S$.
 
\begin{lemma}
\label{lem:5}
Let $p$ be a premonic linear pattern admitted by an ideal $I$ of a numerical semigroup $S$. Then $p(I)$ is an ideal of $S$. 
\end{lemma}
\begin{proof}
Clearly, if $p$ is a pattern admitted by $I$ then $p(I)\subseteq S$. 
If $p(X_1,\dots,X_n)=\sum_{i=1}^na_iX_i+a_0$ is a premonic pattern then $\sum_{i=1}^{n'}a_i=1$ for some $1\leq n'\leq n$.  
If $s_1,\dots,s_n$ is a non-increasing sequence of elements from $I$ and $s\in S$ then $s_1+s,\dots,s_{n'}+s,s_{n'+1},\dots,s_n$ is a non-increasing sequence of elements from $I$ for any $1\leq n'\leq n$. 
We have $p(s_1,\dots,s_n)+s=\sum_{i=1}^{n}a_is_i+a_0+s=\sum_{i=1}^{n}a_is_1+a_0+(\sum_{i=1}^{n'}a_i)s=\sum_{i=1}^{n'}a_i(s_i+s)+\sum_{i=n'+1}^na_is_i+a_0=p(s_1+s,\dots,s_{n'}+s,s_{n'+1},\dots,s_{n})\in p(I)$ for all non-increasing sequences $s_1,\dots,s_n$ of elements from $I$ and for all $s\in S$.
Therefore $p(I)+S\subseteq p(I)$, and $p(I)$ is an ideal of $S$. 
\end{proof}
%In Lemma~\ref{lem:5} there is a motivation for defining patterns as we do here, from $M(S)$ to $M(S)$ and not to $S$. 
%Indeed, if the image of $p(M(S))$ was allowed to contain the zero element, then it would do so for some semigroups and some patterns, but not for all. 
%Since an ideal of $S$ that contains $0$ equals $S$ and we have chosen $M(S)$ as the domain of $p$ and not $S$, we find this inconsistent.   
%For example, the Arf pattern $p(X_1,X_2,X_3)=X_1+X_2-X_3$ satisfies $p(M(S))= M(S)$ and the pattern $p(X_1,X_2,X_3=X_1+3X_2-4X_$
%However, as we will see now, surprisingly the set of problematic patterns turns out to be empty. (Still missing for non-homogeneous patterns.)

%Restricting to the set of monic linear patterns with leading coefficient one, we see therefore that they all correspond to ideals of $S$. 

Note that if $p(X_1,\dots,X_n)=\sum_{i=1}^na_iX_i$ is a premonic linear pattern admitted by the maximal ideal $M(S)$ of a numerical semigroups $S$, then $\gcd(a_1,\dots,a_n)=1$, so that, by Theorem~\ref{thmhomideals}, $p(M(S))$ is either a numerical semigroup contained in $S$ or the maximal ideal of a numerical semigroup contained in $S$. 
But although an ideal of $S$ that contains zero must be equal to $S$, it is not true in general that if $p$ is a premonic, homogeneous linear pattern admitted by $M(S)$, then $p(M(S))\cup\{0\}=S$. 
Consider for example the pattern $X_1+X_2$ with image $2M(S) \subsetneq M(S)$. 
However, as we have already seen, if $p$ is a premonic linear homogeneous pattern admitted by a numerical semigroup $S$, then $p(S)=S$. 
\begin{example}
The Arf pattern $p(X_1,X_2,X_3)=X_1+X_2-X_3$ is a monic linear homogeneous pattern. If $S$ is a numerical semigroup Arf,  
%Indeed, since $p(s,s,s)=s$ for all $s\in S$ it is clear that $I\subseteq p(I)$. 
%Assume that $p(s_1,s_2,s_3)=s_1+s_2-s_3=s$ for some $s_1, s_2, s_3\in I$ and some $s\in S$. Then $s_1+s_2=s+s_3$ 
then $p(S)=S$, and since $p(0,0,0)=0$ and $p^{-1}(0)=(0,0,0)$ also $p(M(S))=M(S)$. 
If $I$ is an ideal of $S$, then $I\subseteq p(I)$, but in general it is not true that  $p(I)\subseteq I$.  
For example, if $S=\left\langle 3,5,7\right\rangle$ and $I=S\setminus\{0,7\}$, then $p(5,5,3)=7\in p(I)\setminus I$ and $p(I)=M(S)$. 
%More generally, if $p$ is a subtraction pattern admitted by a numerical semigroup $S$, so that $p(X_1,\dots,X_n)=X_1+\cdots+X_{n-1}-X_n$, then the image of $p$ satisfies $S= \frac{p(M(S))\cup\{0\}}{(n-2)}$  and $S= \frac{p(S)}{(n-2)}$
%Indeed, if $m$ is the multiplicity of $S$, then $p(s,\dots,s,m,m)=(n-2)s$ for any $s\in M(S)$. 
%Therefore $(n-2)M(S)\subseteq p(M(S))$ so that $M(S)\subseteq \frac{p(M(S))\cup\{0\}}{(n-2)}$. Suppose that 
\end{example}

%In general it is not true that the image of a linear pattern is an ideal of a numerical semigroup. 
%For example, the pattern $p(X_1)=2X_1$ is a pattern for any numerical semigroup $S$, but if $s\in S$ then  $p(s)+s=3s\in p(S)$ implies that $2|s$ which cannot be true for all elements in $S$. Therefore $p(S)+S\subsetneq p(S)$ and $p(S)$ is not an ideal of $S$. 
%However, as we will see now, the image of the sum of two patterns in different variables for which the image is an ideal is also an ideal. 
%\begin{lemma}
%Let $p(X_1,\dots,X_n)$ and $q(Y_1,\dots,Y_m)$ be two patterns admitted by an ideal $I$ of a numerical semigroup $S$ and assume that $p(I)$ and $q(I)$ are two ideals of $S$. 
%Then the image $(p+q)(I)$ of the pattern $(p+q)(X_1,\dots,X_n,Y_1,\dots,Y_m)$ is also an ideal of $S$. 
%\end{lemma}
%\begin{proof}
%If $p(s_1,\dots,s_n)+s\in I'$ and $q(t_1,\dots,t_m)+t\in I''$ for all non-increasing sequences $s_1,\dots,s_n$ and $t_1,\dots,t_m$ with $s_n\geq t_1$ and for all $s,t\in S$, then 

Finally, we show that there is a relation between relative ideals and premonic linear non-homogeneous patterns. 
\begin{lemma}
Let $p(X_1,\dots,X_n)=a_1X_1+\cdots+a_nX_n+a_0$ be a linear non-homogeneous pattern admitted by an ideal $I$ of a numerical semigroup $S$ and let $q(X_1,\dots,X_n)=p(X_1,\dots,X_n)-a_0$ be the homogeneous linear part of $p$. 
If $p$ is premonic (and therefore also $q$), then $q(I)$ is a relative ideal of $S$. 
\end{lemma}
\begin{proof}
Since $p(X_1,\dots,X_n)=q(X_1,\dots,X_n)+a_0$ by Lemma~\ref{lem:5} we have $q(I)+S\subseteq q(I)$ and $q(I)+a_0 \subseteq S$. 
\end{proof}

%{\color{red} Are there relative ideals that are not of the form $q(S)$ for some linear homogeneous polynomial $q$ with leading coefficient one?} %klara

\section{Calculating the image of a pattern}
\label{sec:calculate}
%The results in Section~\ref{sec:patternsofideals} can be used to design algorithms which determines whether a given pattern is admissible, that is, whether there is at least one ideal of a numerical semigroup that admits the pattern. 
%However, given an ideal $I$ of a numerical semigroup, the results in Section~\ref{sec:patternsofideals} do not tell us whether the pattern is admitted by $I$. 
%Instead, 
%The following result tells us that it is possible to determine whether a given pattern $p$ is admitted by a given ideal $I$ of a numerical semigroup $S$ by calculating the image of a finite number of non-increasing sequences of $I$ under $p$. 
%\begin{lemma}
%Let $\mu=\min(I)$ and let $t_i=\min_{s\in I}(s: p(\overbrace{s,\dots,s}^i,\overbrace{\mu,\dots,\mu}^{n-i})>c(I))$. If  $p(s_1,\dots,s_n)\in I$ for all non-increasing sequences $s_1\dots,s_n\in I$ such that $s_i\leq t_i$, then $I$ admits $p$. 
%\end{lemma}
%\begin{proof}
%By Lemma 9 in \cite{MariaPedro}, $\sum_{i=1}^na_is_i\geq \sum_{i=1}^{n'}a_is_{n'}+\sum_{i=n'+1}^{n}a_is_i$ for all $1\leq n'\leq n$. Therefore, if $s_1,\dots,s_n\in I$ is a non-increasing sequence of $I$ and $j$ is the largest index such that $s_j\geq t_i$ ($j=-1$ if $s_j<t$ for all $j$), then  $p(s_1,\dots,s_n)\geq p(\underbrace{t_i,\dots,t_i}_j,s_{j+1},\dots,s_{n})$.   
%\end{proof}

The following result is useful for calculating the image $p(I)$ of an admissible linear homogeneous pattern. 
\begin{lemma}
\label{lemma:calc1}
Let $I$ be an ideal of a numerical semigroup. % and let $c(I)$ be the smallest element in $I$ such that $z\in I$ for all integers $z\geq c$. 
Let $p(X_1,\dots,X_n)=\sum_{i=1}^na_iX_i$ be a homogeneous linear pattern with $\gcd(a_1,\dots,a_n)=d(\geq 1)$ and let $b_1,\dots,b_n$ be (non-unique) integers such that $a_1b_1+\cdots+a_nb_n=d$. 
It is not assumed that $p$ is admitted by $I$. 
Then $p(s_1,\dots,s_n)+d\in p(I)$  for all non-increasing sequences $s_1,\dots,s_n\in I$ such that $s_1+b_1,\dots,s_n+b_n$ is also a non-increasing sequence of elements from $I$.    
\end{lemma}
\begin{proof}
We have $p(s_1,\dots,s_n)+d=p(s_1,\dots,s_n)+p(b_1,\dots,b_n)=p(s_1+b_1,\dots,s_n+b_n)$. Therefore, if $s_1+b_1,\dots,s_n+b_n$ is a non-increasing sequence of elements from $I$, then $p(s_1,\dots,s_n)+d\in p(I)$. 
\end{proof}
Note that any choice of $b_1,\dots,b_n$ such that $a_1b_1+\cdots+a_nb_n=d$ will do. 
In practice it may be useful to instead require that $s_1\geq \cdots \geq s_n\geq c(I)$ and $s_1+b_1\geq \cdots \geq s_n+b_n\geq c(I)$ where $c(I)$ is the smallest element in $I$ such that $z\in I$ for all integers $z\geq c$. Clearly then both $s_1,\dots,s_n$ and $s_1+b_1,\dots,s_n+b_n$ are  non-increasing sequences of elements from $I$. 

\begin{theorem} 
\label{thm:calc}
Let $I$, $c(I)$, $p$, $d$ and $b_1,\dots,b_n$ be as in Lemma~\ref{lemma:calc1}. Then $J=p(I)/d$ is an ideal of a numerical semigroup. Let $c(J)$ be the maximum of the small elements of $J$. Also let $\alpha=\sum_{i=1}^na_i/d$. Then $c(J)< p(s_1,\dots,s_n)/d$ whenever $s_n\geq c(I)-\min(0,(\alpha-1) b_n)$ and $s_i\geq s_j + \max(0,(\alpha-1)(b_j-b_i))$ for $1\leq i< n$.
\end{theorem}    
\begin{proof}
If $d=\gcd(a_1,\dots,a_n)$ then $d$ divides all elements of $p(I)$. Dividing $p(I)$ by $d$ gives the image of $I$ under the pattern $q=p/d=\sum_{i=1}^n\frac{a_i}{d}X_i$ which has relatively prime coefficients $c_i=\frac{a_i}{d}$ such that $c_1b_1+\cdots+c_nb_n=1$ and $\sum_{i=1}^nc_i=\alpha$. Therefore, by Theorem~\ref{thmhomideals}, $J=q(I)$ is an ideal of some semigroup. 
  
Any non-increasing sequence $s_1,\dots,s_n\in \mathbb{Z}$ with $s_n\geq c(I)$ is a non-increasing sequence of elements of $I$. 
Note that the $b_i$:s can be negative integers. 
Take $s_n\geq c(I)-\min(0,(\alpha-1) b_n)$ and $s_i\geq s_j + \max(0,(\alpha-1)(b_j-b_i))$ for $1\leq i< n$. 
Under these conditions we have that $s_i+tb_i\geq s_j+tb_j\geq c(I)$ for all $1\leq i\leq j\leq n$ and $0\leq t\leq \alpha-1$, so that $q(s_1+tb_1,\dots,s_n+tb_n)\in q(I)$. 
Now note that $q(s_1+x+(t+1)b_1,\dots,s_n+x+(t+1)b_n)=q(s_1+x+tb_1,\dots,s_n+x+tb_n)+1$ for all $0\leq t\leq \alpha-1$ and for all $x\geq 0$. 
Also, $q(s_1+x,\dots,s_n+x)=q(s_1,\dots,s_n)+\alpha x$. 
Therefore $q(I)$ contains all integers larger or equal to $q(s_1,\dots,s_n)$ with $s_n\geq c(I)-\min(0,(\alpha-1) b_n)$ and $s_i\geq s_j + \max(0,(\alpha-1)(b_j-b_i))$ for $1\leq i< n$.    
\end{proof}

Theorem~\ref{thm:calc} implies that the set of non-increasing sequences of $I$ which is needed for calculating explicitly $p(I)$ is finite. However, in practice this number will depend on the choice of $b_1,\dots,b_n$.

A linear pattern $p(X_1,\dots,X_n)=\sum_{i=1}^na_iX_i+a_0$ is called \emph{strongly admissible} if the partial sums $\sum_{i=1}^{n'}a_i\geq 1$ for all $1\leq n'\leq n$. (Strongly admissible patterns were introduced differently in \cite{MariaPedro}, but the two definitions are equivalent). 

\begin{lemma}
\label{lemma:Y}
Let $C$ be a positive integer constant, and let $p(X_1,\dots,X_n)=\sum_{i=1}^na_iX_i+a_0$ be a strongly admissible linear pattern.  
Let $Y_t=\{(t,s_2,\dots,s_n): \forall~ 2\leq i \leq n~ s_i\in I, \forall~ 2\leq i \leq n-1 ~ s_i\geq s_{i+1}\}$ and let $Y(C)=\bigcup_{t\leq x,~t\in I} Y_t$  for the smallest $x\in I$ such that $p(x,s_2,\dots,s_n)\geq C$ for all $s_2,\dots,s_n\in I$ such that $x\geq s_1 \geq \cdots \geq s_n$. 
Then $Y(C)$ is a well-defined finite set and contains the set of non-increasing sequences $s_1,\dots,s_n\in I$ such that $p(s_1,\dots,s_n)<C$. 
\end{lemma}
\begin{proof}
First we prove that if  $\sum_{i=1}^{n'}a_i\geq 1$ for all $1\leq n'\leq n$, then there is an $x\in I$ such that $p(x,s_2,\dots,s_n)\geq C$ for all $s_2,\dots,s_n\in I$ such that $x\geq s_2 \geq \cdots \geq s_n$. 
Let $s_1\geq \cdots \geq s_n\geq 0$. 
Then $p(s_1,\dots,s_n)=\sum_{i=1}^na_is_i+a_0=\sum_{i=1}^{n-1}(\sum_{j=1}^ia_j)(s_i-s_{i+1})+\sum_{j=1}^na_js_n+a_0\geq \sum_{i=1}^{n-1}1\cdot (s_i-s_{i+1})+1\cdot s_n+a_0=s_1+a_0$. 
Therefore, by taking $x=s_1\geq C-a_0$, one has $p(x,s_2,\dots,s_n)\geq C$ for all $s_2,\dots,s_n\in I$ such that $x\geq s_2\geq \cdots \geq s_n$, so $Y$ is a well-defined finite set. 

Now note that we have also proved that for all $s_1,\dots,s_n\in I$ with $s_1>x$, $p(s_1,\dots,s_n)\geq C$. Consequently, if $s=(s_1,\dots,s_n)$ is a  non-increasing sequence of elements of $I$ such that $p(s_1,\dots,s_n)<C$, then $s\in Y$. 

%Now note that if $s=(s_1,\dots,s_n)$ is a non-increasing sequence of elements from $I\setminus Y$, then either $s_1=s_n$ and then $p(s_1,\dots,s_n)=p(s_1,\dots,s_1)=\sum_{i=1}^na_is_1\geq s_1\geq x\geq C$ or $s_1>s_n$, in which case $s \geq_{\circ} s'$ for some $s'\in Y_x$. By Lemma~\ref{lemma:linext} $\geq_p$ is a linear extension of $\geq_{\circ}$, implying that $p(s)\geq_{\circ} p(s')$. By definition of $x$ and $Y_x$ we have $p(s')> C$, therefore $p(s)\geq C$.  
%Consequently, if $s=(s_1,\dots,s_n)$ is a  non-increasing sequence of elements of $I$ such that $p(s_1,\dots,s_n)<C$, then $s\in Y$. 
\end{proof}

The following algorithm calculates $p(I)$ by calculating first an upper bound $C\geq c(J)$ and then calculating $p(s)$ for all $s\in Y(C)$.  
\begin{lemma}
\label{lemma:alg}
Let notation be as in Lemma~\ref{lemma:calc1} and Theorem~\ref{thm:calc}. Assume also that the admissible homogeneous pattern polynomial $p(X_1,\dots,X_n)=\sum_{i=1}^na_iX_i$ is strongly admissible. %satisfies $\sum_{i=1}^{n'}a_i\geq 1$ for all $1\leq n'\leq n$.  
%Let $q=p/d$ and calculate $p(I)$ as $dq(I)$. 
The following algorithm can be used to calculate $p(I)$. .   
\begin{enumerate}
\item Set $q=p/d$. 
\item Calculate $C=q(s_1,\dots,s_n)$ with $s_n= c(I)+\min(0,\alpha b_n)$ and $s_i= s_j + \min(0,\alpha(b_j-b_i))$ for $1\leq i< n$. 
\item Calculate $Q:=\{q(s_1,\dots,s_n): (s_1,\dots,s_n)\in Y(C)\}$ where $Y(C)$ is the set defined in Lemma~\ref{lemma:Y}. 
\begin{comment}
The following algorithm calculates $Q=q(Y(C))$:\bigskip\\
%For $1\leq i\leq n$:  \\
%\hspace*{1cm} Let $x:=\mu(I)=\min(I)$;\\
%\hspace*{1cm} While $p(\overbrace{x,\dots,x}^i,\overbrace{\mu(I),\dots,\mu(I)}^{n-i})<C$: \\
%\hspace*{2cm} $x:=\min(s\in I: s>x)$; \\
\hspace*{1cm} $Q=\{\}$;\\ 
\hspace*{1cm} $s_1:=\mu(I)=\min(I)$;\\
\hspace*{1cm} $end:=true$;\\
\hspace*{1cm} Do\\
\hspace*{2cm} For $\mu(I)\leq s_{2}\leq s_1$ and $s_2\in I$\\
\hspace*{3cm} For $\mu(I)\leq s_{3}\leq s_{2}$ and $s_3\in I$\\
\hspace*{4cm} $\vdots$\\
\hspace*{4cm} For $\mu(I)\leq s_n\leq s_{n-1}$ and $s_n\in I$\\
\hspace*{5cm} $a:=p(s_1,\dots,s_n)$;\\
\hspace*{5cm} Add $a$ to $Q$;\\
\hspace*{5cm} If $a<c(J)$\\ 
\hspace*{6cm} $end:=false$; \\
\hspace*{2cm} $s_1:=\min(s\in I:s\geq s_1)$;\\
\hspace*{1cm} while not $end$;\\
\hspace*{1cm} Return $Q$;

%For $1\leq i\leq n$:  \\
%\hspace*{1cm} Let $x:=\mu(I)=\min(I)$;\\
%\hspace*{1cm} While $p(\overbrace{x,\dots,x}^i,\overbrace{\mu(I),\dots,\mu(I)}^{n-i})<C$: \\
%\hspace*{2cm} $x:=\min(s\in I: s>x)$; \\
%\hspace*{1cm} For $\mu(I)\leq s_{i+1}\geq x$:\\
%\hspace*{2cm} For $\mu(I)\leq s_{i+2}\leq s_{i+1}$:\\
%\hspace*{3cm} $\vdots$\\
%\hspace*{3cm} For $\mu(I)\leq s_n\leq s_{n-1}$:\\
%\hspace*{4cm} Calculate $p(\underbrace{x,\dots,x}_i,s_{i+1},\dots,s_n)$; 

\end{comment}
\item Now $q(I)=Q\cup \{z\in \mathbb{Z}: z\geq C\}$ and $p(I)=\{ds: s\in q(I)\}$.
\end{enumerate}
\end{lemma}  
\begin{proof}
By combining Theorem~\ref{thm:calc} and Lemma~\ref{lemma:Y}.

\end{proof}
%Just evaluate $p$ in all non-increasing sequences $s_1,\dots,s_n\in I$ such that 

Given a linear strongly admissible pattern polynomial and two ideals $I$ of a numerical semigroup $S$ and $J$ of a numerical semigroup $S'$, step 3 in the algorithm in Lemma~\ref{lemma:alg} alone can be used to determine whether or not $p(I)\subseteq J$, after defining $C=\min(s\in J: n\in J~ \forall~ n\in \{z\in \mathbb{Z}: z\geq s\})$. 
In the particular case when $I$ is an ideal of a numerical semigroup $S$ and $J=S$ (and so $C$ is the conductor of $S$), this calculation determines whether or not $I$ admits $p$.  The existence of an algorithm that determines if a strongly admissible pattern is admitted by a numerical semigroup was first announced in \cite{MariaPedro}. 
\begin{lemma}
\label{lemma:calc2}
Let $I$ be an ideal of a numerical semigroup, let $p(X_1,\dots,X_n)=\sum_{i=1}^na_iX_i+a_0$ be an admissible linear pattern and let $$q(X_1,\dots,X_n)=p(X_1,\dots,X_n)-a_0.$$  
Then $p(I)=q(I)+a_0$.  
\end{lemma}
\begin{proof}
Indeed, any element in $p(I)$ is of the form $p(s_1,\dots,s_n)=\sum_{i=1}^na_is_i+a_0=q(s_1,\dots,s_n)+a_0$. 
\end{proof}
Together Lemma~\ref{lemma:alg} and Lemma~\ref{lemma:calc2} can be used to calculate the image of an ideal of a numerical semigroup under a linear strongly admissible pattern $p$ in a finite number of steps. 

\section{Patterns of ideals of numerical semigroups}
\label{sec:patternsofideals}
In this article a pattern admitted by an ideal $I$ of a numerical semigroup $S$ is a multivariate polynomial function which evaluated on non-increasing sequences of elements from $I$ returns an element of $S$. 
This definition generalises previous definitions of patterns admitted by numerical semigroups. 
Indeed, a homogeneous linear pattern as defined in \cite{MariaPedro} is according to our definition still a pattern admitted by a numerical semigroup. 
However, a non-homogeneous linear pattern as defined in \cite{MariaPedroAlbert} is now a pattern admitted by the maximal ideal of some numerical semigroup. 

The concept can be generalised further, for example by relaxing the criteria that the codomain of a pattern admitted by an ideal necessarily should be a numerical semigroup containing the ideal. Then the codomain of the pattern can be another numerical semigroup, or generalising even more, an ideal of some numerical semigroup.  
It is possible to go even further by considering relative ideals instead of ideals.  
One can also restrict to patterns with some particular property like for example linearity or homogeneity.

We say that a linear pattern that returns an element in $I$ when evaluated on the non-increasing sequences of elements of $I$ is an \emph{endopattern} of $I$. 
A pattern admitted by an ideal $I$ with codomain $J$ is \emph{surjective} when $p(I)=J$. 
A surjective endopattern of $I$ is therefore a pattern $p$ such that $p(I)=I$. 
%If additionally the image $p(I)$ equals $I$, then we say that $p$ is a \emph{surjective endopattern of $I$}. 
Formally, an endopattern of $I$ is an endomorphism of the set of non-increasing sequences of $n$ elements from $I$. 
Note that, for example, the map $x\mapsto (x,\dots,x)$ is an embedding of the image of the pattern in the set of non-increasing sequences of $n$ elements from $I$.

In this article, the focus is on linear endopatterns of numerical semigroups and ideals of numerical semigroups, in particular maximal ideal. To avoid confusion we will each time explicitly state the properties of the patterns that we consider in each moment. 

We start with necessary conditions for linear patterns to be endopatterns and surjective endopatterns of numerical semigroups.  
We will repeatedly make use of the following result, first proved in \cite{MariaPedro} for homogeneous linear patterns and in \cite{MariaPedroAlbert} for non-homogeneous linear patterns. Here we prove the result for ideals of numerical semigroups, making use of Abel's partial summation formula (this proof is due to Christian Gottlieb). 
\begin{lemma}
\label{lemma:MPA}
If $p(X_1,\dots,X_n)=\sum_{i=1}^na_iX_i+a_0$ is a linear pattern admitted by an ideal $I$ of a numerical semigroup $S$ (i.e. $p(s_1,\dots,s_n)\in S$ for all non-increasing sequences $s_1,\dots,s_n$ of elements from $I$), then  
\begin{itemize}
\item $\sum_{i=1}^{n'}a_i\geq 0$ for all $1\leq n'\leq n$, and 
\item $\sum_{i=1}^na_is_i\geq \sum_{i=1}^na_is_n$ for all non-increasing sequences $s_1,\dots,s_n\in I$.  
\end{itemize}
\end{lemma}
\begin{proof}
Assume there is an $n'$ with $1\leq n'\leq n$ such that $\sum_{i=1}^{n'}a_i< 0$. If $s$ is very large compared to $t$ we obtain $$p(\underbrace{s,\dots,s}_{n'},\underbrace{t,\dots,t}_{n-n'})=\left(\sum_{i=1}^{n'}a_i\right)s+ \left(\sum_{i=n'+1}^na_i\right)t<0,$$ implying that $p$ cannot be a pattern admitted by $I$. 

Now, let $A_{j}=\sum_{i=1}^{j}a_i$. By Abel's formula for summation by parts \cite{Apostol}, we have $\sum_{i=1}^na_i(s_i-s_n)=A_n(s_n-s_n) -\sum_{i=1}^{n-1}A_i(s_{i+1}-s_i)=-\sum_{i=1}^{n-1}A_i(s_{i+1}-s_i)\geq 0$.  
\end{proof} 

\begin{proposition}
\label{lemmaendS}
A linear endopattern of a numerical semigroup $S$ is simply a linear pattern defined by a polynomial $p(X_1,\dots,X_n)=\sum_{i=1}^na_iX_n+a_0$ admitted by $S$. 
Therefore 
\begin{itemize}
\item $\sum_{i=1}^{n'}a_i\geq 0$ for all $1\leq n'\leq n$, and 
\item $a_0\in S$. 
\end{itemize}
\end{proposition}
\begin{proof}
Let $p$ be a linear endopattern of $S$, then $p$ is a pattern defined by a linear multivariate polynomial $p(X_1,\dots,X_n)=\sum_{i=1}^na_iX_i+a_0$  such that evaluated on any non-increasing sequence of elements of $S$, the result is in $S$. 
In particular, $p(0,\dots,0)=a_0\in S$. 
For the rest of the statement, apply Lemma \ref{lemma:MPA}. 
%Finally, it was proved in \cite{MariaPedro,MariaPedroAlbert} that if  $p$ is a linear pattern with image in $S$, then $\sum_{i=1}^{n'}a_i\geq 0$ for all $1\leq n'\leq n$. 
%Indeed, otherwise $p(\underbrace{s,\dots,s}_{n'},\underbrace{t,\dots,t}_{n-n'})=\left(\sum_{i=1}^{n'}a_i\right)s+ \left(\sum_{i=n'+1}^na_i\right)t<0$ if $s$ is taken very large compared to $t$, implying that $p$ cannot be a pattern of a numerical semigroup (nor of an ideal of a numerical semigroup).  
\end{proof}

\begin{proposition}
\label{lemmaautS}
A linear surjective endopattern $p$ of a numerical semigroup $S$ is necessarily homogeneous. 
If $p$ is a premonic homogeneous endopattern of $S$, then $p$ is always surjective. 
\end{proposition}
\begin{proof}
A surjective endopattern $p$ of $S$ is an endopattern of $S$, therefore, by Lemma~\ref{lemmaendS}, $p$ is a linear pattern defined by a polynomial of the form $p(X_1,\dots,X_n)=\sum_{i=1}^na_iX_i+a_0$ with $a_0\in S$ and  $\sum_{i=1}^{n'}a_i\geq 0$ for all $1\leq n'\leq n$. 
But if $a_0>0$ then this gives $p(s_1,\dots,s_n)>0$ for all non-increasing sequences of $S$ so that $ p(S)\subsetneq S$. 
Consequently, if $p(S)=S$, then $p$ is defined by a homogeneous linear pattern. 
Finally, if $p$ is premonic then, by Lemma~\ref{lemmapremonic}, $p$ is surjective. 
\end{proof} 

%Remember that Lemma~\ref{lemmapremonic} says that if $p$ is a premonic homogeneous endopattern of a numerical semigroup $S$, then $p$ is surjective. 

%\begin{proof}
%If $p(X_1,\dots,X_n)=\sum_{i=1}^na_iX_i$ is premonic, then there is some $j$ for which $\sum_{i=1}^ja_i=1$. Therefore 

%The difference between endopatterns of numerical semigroups and endopatterns of maximal ideals is of course that in the latter the zero element is not contained neither in the domain nor in  the codomain. 

The next result gives a necessary condition for when a polynomial defines a linear pattern admitted by a proper ideal of a numerical semigroup. 
When the ideal is a maximal ideal then this result strengthens the necessary condition given in \cite{MariaPedroAlbert}. 
\begin{lemma}
\label{neccondreinf}
  If $S$ is a numerical semigroup and  $p=\sum_{i=1}^na_iX_i+a_0$ is a linear pattern admitted by a proper ideal $I$ of $S$, then 
\begin{itemize}
\item $\sum_{i=1}^{n'}a_i\geq 0$ for all $1\leq n'\leq n$ and, moreover, 
\item $\sum_{i=1}^na_i\geq \max(0,-a_0/\mu(I))$, where $\mu(I)=\min(I)$.
\end{itemize} 
\end{lemma}
\begin{proof}
For the first part, apply Lemma \ref{lemma:MPA}. 
%The proof of $\sum_{i=1}^{n'}a_i\geq 0$ for all $1\leq n'\leq n$ is analogous to the proof of the first part of Proposition~\ref{lemmaendS} (originally \cite{MariaPedroAlbert}).
The second part is an improvement of $\sum_{i=1}^na_i\geq 0$ which is relevant only when $a_0<0$. 
There are no linear patterns admitted by $S$ with $a_0<0$ (see Proposition~\ref{lemmaendS}), but there may be linear patterns admitted by $I$ with that property. 
Therefore assume that $p(X_1,\dots,X_n)=\sum_{i=1}^na_iX_i+a_0$ is a linear pattern admitted by $I$ with $a_0<0$ and $\sum_{i=1}^na_i<\max(0,-a_0/\mu(I))=-a_0/\mu(I)$, then $p(\mu(I),\dots,\mu(I)))=\sum_{i=1}^na_i\mu(I)+a_0<\left(-a_0/\mu(I)\right)\cdot \mu(I)+a_0=0$ so that $p(\mu(I),\dots,\mu(I))\not\in S$. But then $p$ cannot be a pattern of $I$ and we have a contradiction. 
\end{proof} 

We now use Lemma~\ref{neccondreinf} to give necessary conditions for when a pattern is an endopattern of a proper ideal of a numerical semigroup. 

%Note that any proper ideal $I$ of a numerical semigroup $S$ is contained in $M(S)$, so that any pattern admitted by $M(S)$ is also admitted by $I$. The next result combines this fact with the results from Lemma~\ref{neccondreinf} and Lemma~\ref{lemma0} about linear endopatterns of maximal ideals.  
\begin{proposition}
\label{propendM}
A linear endopattern of a proper ideal $I$ of a numerical semigroup $S$ is a linear pattern $p(X_1,\dots,X_n)=\sum_{i=1}^na_iX_i+a_0$ admitted by $I$, and so (by Lemma~\ref{neccondreinf}) $p$ necessarily satisfies 
\begin{itemize}
\item $\sum_{i=1}^{n'}a_i\geq 0$ for all $1\leq n'<n$, 
\item $\sum_{i=1}^na_i\geq \max(0,-a_0/\mu(I))$,
\end{itemize}
and additionally 
\begin{itemize}
\item $a_0>0$ or $\sum_{i=1}^{n}a_i>\max(0,-a_0/\mu(I))$ (or both),
\end{itemize}  
where $\mu(I)=\min(I)$. 
\end{proposition}
\begin{proof}
The first part of this result is Lemma~\ref{neccondreinf}. 
For the second part, assume that $a_0\leq 0$ and $\sum_{i=1}^{n}a_i=\max(0,-a_0/\mu(I))$. 
%If additionally $\sum_{i=1}^na_i>\max(0,-a_0/m(S))$ then $p(s_1,\dots,s_n)=\sum_{i=1}^na_is_i+a_0 \geq \sum_{i=1}^na_is_n +a_0 \geq \sum_{i=1}^na_im +a_0 = \left(\sum_{i=1}^na_i\right)m+a_0 > \max(0,-a_0/m) m+a_0=\max(a_0,0)\geq 0$ for all non-increasing sequences $s_1,\dots,s_n\in M(S)$ and with $m=m(S)$, so that $p(M(S))\subseteq M(S)$. Also, if $a_0>0$ and $\sum_{i=1}^na_i=\max(0,-a_0/m)=0$ then similarly $\sum_{i=1}^na_is_i+a_0\geq \left(\sum_{i=1}^na_i\right)s_n+a_0 =0+a_0>0$, so that $p(M(S))\subseteq M(S)$. 
Then $\sum_{i=1}^na_i\mu(I)+a_0=\max(0,-a_0/\mu(I))\mu(I)+a_0=0$, so that $p(I)\not \subseteq I$.
\end{proof}

\begin{proposition}
\label{propendomax}
A linear pattern $p(X_1,\dots,X_n)=\sum_{i=1}^na_iX_i+a_0$ admitted by the maximal ideal $M(S)$ of a numerical semigroup $S$ is an endopattern of $M(S)$ if and only if $a_0>0$ or $\sum_{i=1}^na_i>\max(0,-a_0/m(S))$ (or both).
\end{proposition}
\begin{proof}
By Proposition~\ref{propendM}, if $p$ is an endopattern, then $a_0>0$ or $\sum_{i=1}^na_i>\max(0,-a_0/m(S))$ (or both). 

Now assume that  $p(X_1,\dots,X_n)=\sum_{i=1}^na_iX_i+a_0$ is a linear pattern admitted by $M(S)$.  
Then by Lemma~\ref{lemma:MPA}, $\sum_{i=1}^n a_is_i\geq \sum_{i=1}^na_is_n\geq 0$ for all non-increasing sequences $s_1,\dots,s_n\in M(S)$. 
%(Proved by induction on $n$. 
%It is clearly true for $n=1$. 
%Assume true for $n=k$, then $\sum_{i=1}^{k+1}a_is_i = \sum_{i=1}^{k}a_is_i+a_{k+1}s_{k+1} \geq \sum_{i=1}^{k}a_is_k+a_{k+1}s_{k+1}\geq  \sum_{i=1}^{k}a_is_{k+1}+a_{k+1}s_{k+1}=\sum_{i=1}^ka_is_{k+1}$. )
Therefore, if $a_0>0$, then $p(s_1,\dots,s_n)=\sum_{i=1}^na_is_i+a_0\geq a_0>0$ for all non-increasing sequences $s_1,\dots,s_n\in M(S)$, so that $p(M(S))\subseteq M(S)$ and $p$ is an endopattern of $M(S)$. 
Also, if  $\sum_{i=1}^na_i>\max(0,-a_0/m(S))$, then $p(s_1,\dots,s_n)=\sum_{i=1}^na_is_i+a_0\geq (\sum_{i=1}^na_i)m(S)+a_0>\max(0,-a_0/m(S))m(S)+a_0\geq 0$ for all non-increasing sequences $s_1,\dots,s_n\in M(S)$, so that $p(M(S))\subseteq M(S)$ and $p$ is an endopattern of $M(S)$. 
\end{proof}

\begin{proposition}
\label{prop:surendo}
Any linear surjective endopattern of a proper ideal $I$ of a semigroup $S$ is necessarily of the form  $p(X_1,\dots,X_n)=\sum_{i=1}^na_iX_i+a_0$ satisfying $a_0=-(\sum_{i=1}^na_i-1)\mu(I)$ where $\mu(I)$ is the smallest element of $I$. Also, if $p$ is a premonic endopattern of $I$, such that $a_0=-(\sum_{i=1}^na_i-1)\mu(I)$, then $p$ is surjective.
\end{proposition}
\begin{proof}
Denote by $\mu(I)$ the smallest element of $I$. If $p=\sum_{i=1}^na_iX_i+a_0$ is a linear surjective endopattern of $I$, then by Lemma \ref{lemma:MPA} $\sum_{i=1}^na_i s_i\geq  \sum_{i=1}^na_i s_n \geq \sum_{i=1}^na_i \mu(I)$ for all non-increasing sequences of $I$. Since $p$ is surjective, $\mu(I)$ is in $p(I)$, forcing  $\sum_{i=1}^{n}a_i \mu(I)+a_0=\mu(I)$ so that $a_0=-(\sum_{i=1}^{n}a_i-1)\mu(I)$.
 
Now, if $p$ is premonic, then $\sum_{i=1}^ja_i=1$ for some $j\leq n$ so that if $a_0=-(\sum_{i=1}^na_i -1)\mu$, then $a_0=-\sum_{i=j+1}^na_i\mu$ and so $p(s,\dots,s,\mu,\dots,\mu)=\sum_{i=1}^ja_is+\sum_{i=j+1}^na_i\mu+a_0= s$ for all $s\in I$, implying that $p(I)=I$.   
%If $a_0\geq 0$, then this implies that $\sum_{i=1}^na_i=1$ and $a_0=0$. If $a_0<0$ then $a_0=-\left(\sum_{i=1}^na_i -1\right) \mu(I)$. However this condition is not sufficient.                         
\end{proof}

%\begin{lemma}
%A linear surjective endopattern of a maximal ideal $M(S)$ of a numerical semigroup $S$ is a linear pattern $p(X_1,\dots,X_n)=\sum_{i=1}^na_iX_i+a_0$ such that 
%\end{lemma}
%By Lemma~\ref{neccondreinf} and Lemma~\ref{lemmaendM} either $a_0>0$ and $\sum_{i=1}^na_i\geq 0$ or  $a_0\leq 0$ and $\sum_{i=1}^na_i> -a_0/m(S)$. 
%5If $\sum_{i=1}^na_i>-a_0/m(S)$, then $\sum_{i=1}^na_im(S)> -a_0
%For each $s\in S$ there must be a non-increasing sequence $s_1,\dots,s_n\in M(S)$ such that $\sum_{i=1}^na_is_i+a_0=s$. Since $\sum_{i=1}^{n'}a_i\geq 0$ for all $n'\leq n$, we have $\sum_{i=1}^na_is_i\geq \sum_{i=1}^na_is_n$. 

%\begin{lemma}
%A linear automorphism of a maximal ideal $M(S)$ of a numerical semigroup $S$ is a linear pattern $p(X_1,\dots,X_n)=\sum_{i=1}^na_iX_i$ such that 
%\begin{proof}
%For $p$ to be an automorphism of $M(S)$ there must be a non-increasing sequence $s_1,\dots,s_n$ of elements in $M(S)$ such that $\sum_{i=1}^na_is_i+a_0=m(S)$ where $m(S)$ is the multiplicity of $S$. 
%By Lemma~\ref{lemmaendM}, since $p$ is an endomorphism, $\sum_{i=1}^{n}a_i> \max(0,-a_0/m)$. 
%Also, $\sum_{i=1}^{n'}a_i\geq 0$ for all $n'\leq n$  \cite{MariaPedro,MariaPedroAlbert}.
%But this implies that the sequence that gives the smallest number when evaluted under the pattern is the sequence $m,\dots,m$, so that $\sum_{i=1}^{n}a_is_i\geq \sum_{i=1}^na_im> \max(0,-a_0)$. 

The linear patterns considered in the literature before this article  are either homogeneous patterns admitted by $S$ or non-homogeneous patterns admitted by $M(S)$. They all have the numerical semigroup $S$ as codomain. 
%In order to further link previous work with this, it is interesting to clarify which of these linear patterns are also endopatterns of the maximal ideal $M(S)$. 
The next result shows that almost all these patterns are also endopatterns of $M(S)$. 
\begin{corollary}
\label{cor0}
Let $S$ be a numerical semigroup and $M(S)$ its maximal ideal. 
\begin{itemize}
\item If $p(X_1,\dots,X_n)=\sum_{i=1}^na_iX_i$ is a homogeneous linear pattern admitted by $S$ which is not an endopattern of $M(S)$, then $\sum_{i=1}^na_i=0$. 
\item If $p(X_1,\dots,X_n)=\sum_{i=1}^na_iX_i+a_0$ is a non-homogeneous linear pattern admitted by $M(S)$ which is not an endopattern of $M(S)$, then $a_0\leq 0$ and $\sum_{i=1}^na_i=-a_0/m(S)$. 
\end{itemize}
\end{corollary}
\begin{proof}
\begin{itemize}
\item If $p$ is a homogeneous linear pattern admitted by $S$, then $p$ is also admitted by $M(S)$. By Lemma~\ref{neccondreinf} and Proposition~\ref{propendomax}, if $p$ is not an endopattern then $\sum_{i=1}^n a_i=0$. 
\item If $p$ is a non-homogenous linear pattern admitted by $M(S)$, then by Lemma~\ref{neccondreinf}, Proposition~\ref{propendM} and Proposition~\ref{propendomax}, if $p$ is not an endopattern then $a_0\leq 0$ and $\sum_{i=1}^na_i=-a_0/m(S)$. 
\end{itemize}
\end{proof}

Examples of patterns admitted by a maximal ideal $M(S)$ of a semigroup $S$ that are not endopatterns of $M(S)$ can be found in the two non-homogeneous patterns in Weierstrass semigroups mentioned in the introduction. 

Corollary~\ref{cor0} shows that many of the important patterns previously considered in the literature are endopatterns of $M(S)$. For example, this is true for the Arf pattern, the subtraction patterns and the patterns of the form $X+a$ with $a$ pseudo-Frobenius. 
They all belong to the important class of monic linear patterns. 
%A \emph{monic linear pattern} is defined by a linear polynomial $p(X_1,\dots,X_n)=\sum_{i=1}^na_iX_i+a_0$ with $a_1=1$.   
%\begin{lemma}
%There are no monic linear homogeneous patterns $p(X_1,\dots,X_n)=\sum_{i=1}^na_iX_i$ from $S$ to $S$ with $\sum_{i=1}^n a_i=0$. 
%As a consequence, all monic linear homogeneous patterns from $S$ to $S$ have their image contained in $M(S)$.
%\end{lemma}
%\begin{proof}
%Let $p$ be a monic linear homogeneous pattern with $\sum_{i=1}^na_i=0$. 
%Then for all $s'\geq s \in S$ we have $p(s',s,\dots,s)=s'+\sum_{i=2}^na_is = (s'-s) + p(s,\dots,s)=s'-s\in S$. In particular if $s=m$ is the multiplicity of $S$  and $s'$ is the smallest elements of $S$ that is not a multiple of $m$, then we have $s'-s\in S$ but $s'-s< s'$ and not a multiple of $m$, which is a contradiction. 
%\end{proof}
%{\color{red} What about non-homogeneous patterns?}\\
\begin{lemma}
\label{lemmamonic}
Let $S$ be a numerical semigroup. 
If $S\neq \mathbb{Z}_+$ and $a_0\not\in M(S)$, then there are no monic linear patterns $p(X_1,\dots,X_n)=\sum_{i=1}^na_iX_i+a_0$ admitted by $S$ or by its maximal ideal $M(S)$ with $\sum_{i=1}^n a_i=\max(0,-a_0/m(S))$. 
%As a consequence, if $S\neq \mathbb{Z}_+$, then any monic linear pattern admitted by $M(S)$ is an endopattern of $M(S)$. .
\end{lemma}
\begin{proof}
Let $m=m(S)$ and let $p$ be a monic linear pattern with $\sum_{i=1}^na_i=\max(0,-a_0/m)$. 
%If $a_0>0$, then we know from Lemma~ref{?} that $p$ is an endopattern of $M(S)$. 

If $a_0\leq 0$, then $\max(0,-a_0/m)=-a_0/m$. Let $s$ be the smallest element of $M(S)$ that is not a multiple of $m$, then  $p(s,m,\dots,m)=s+\sum_{i=2}^na_im+a_0 = (s-m) + p(m,\dots,m)=s-m + \max(0,-a_0/m)m+a_0=s-m \in S$. 
Now $s-m< s$ and $s-m$ is not a multiple of $m$, but $s$ is the smallest element in $M(S)$ that is not a multiple of $m>1$, so there is a contradiction. So there are no monic linear patterns admitted by $S$ (or $M(S)$) with $\sum_{i=1}^na_i=\max(0,-a_0/m)$ and $a_0\leq 0$. 

If $a_0> 0$, then $\max(0,-a_0/m)=0$, implying that for all $s\in M(S)$ we have  $p(s,\dots,s)=\sum_{i=1}^na_i s +a_0= a_0\in S$, and since $a_0>0$, we have $a_0\in M(S)$. 
Therefore, there are no monic linear patterns admitted by $S$ (or $M(S)$) with $\sum_{i=1}^na_i=\max(0,-a_0/m)$, $a_0>0$ and $a_0\not\in M(S)$.
\end{proof}

\begin{corollary}
\label{cor:monic}
If $S\neq \mathbb{Z}_+$, then any monic linear pattern admitted by $M(S)$ is an endopattern of $M(S)$. 
\end{corollary}
\begin{proof}
If $p$ is a monic linear pattern admitted by $M(S)$, then, by Lemma~\ref{neccondreinf} and Lemma~\ref{lemmamonic}, $\sum_{i=1}^na_i> \max(0,-a_0/m(S))$. 
Therefore, by Proposition~\ref{propendomax}, $p$ is an endopattern of $M(S)$.  
%
%If $a_0\geq 0$, then $\max(0,-a_0/m(S))=0$, implying that, since $\sum_{i=1}^na_is_i+s_0\geq \sum_{i=1}^na_im(S)+s_0$,  for all non-increasing sequences $s_1,\dots,s_n\in M(S)$, and $m(S)>0$, we have  $\sum_{i=1}^na_im(S)+a_0>0$, so that $p(M(S))\subseteq M(S)$. 
%Then $\sum_{i=1}^na_is_i +a_0\geq \sum_{i=1}^na_im(S)+a_0 >
%
%By Proposition~\ref{propendM}, every linear pattern admitted by $M(S)$ that satisfies $\sum_{i=1}^na_i> \max(0,-a_0/m)$ is an endopattern of $M(S)$. 
%Assume that $p=\sum_{i=1}^na_iX_i+a_0$ is a monic linear pattern admitted by $M(S)$ with $\sum_{i=1}^na_i=\max(0,-a_0/m)$. We just proved in Lemma~\ref{} that if $p$ is monic and  $\sum_{i=1}^na_i=\max(0,-a_0/m)$ then $a_0\in M(S)$. 
%But if $a_0\in M(S)$, then $a_0>0$ so that $\sum_{i=0}^na_is_i\geq 0$ for all non-increasing sequences $s_1,\dots,s_n\in S$, so that $p(s_1,\dots,s_n)=\sum_{i=0}^na_is_i+a_0>0$, and consequently $p(M(S))\subseteq M(S)$.  
\end{proof}

%In Section~\ref{sec:1} we will...

\section{Closures of ideals with respect to linear patterns}
\label{sec:closure}
A pattern is admissible if it is admitted by some numerical semigroup. 
In \cite{MariaPedro} the closure of a numerical semigroup $S$ with respect to an admissible homogeneous pattern $p$ was defined as the smallest numerical semigroup that admits $p$ and contains $S$. 
Here this definition is generalised to non-homogeneous patterns and to ideals of numerical semigroups. 
\begin{definition}
\label{def:closure}
Given an ideal $I$ of a numerical semigroup $S$ and an admissible  pattern $p$ not necessarily admitted by $I$, define the closure of $I$ with respect to $p$ as the smallest ideal $\tilde{I}$ of some numerical semigroup  $\tilde{S}$ that admits $p$ and contains $I$. 

It is not required that $I$ is an ideal of $\tilde{S}$, nor is it required that $\tilde{I}$ is an ideal of $S$. However, by definition, it is always true that $I\subseteq \tilde{I}\subseteq \tilde{S}$. 
\end{definition}

Note that if $I$ is not contained in any ideal of a numerical semigroup that admits $p$, then the closure of $I$ with respect to $p$ will fail to exist. 
This is not a problem for homogeneous linear patterns since a homogeneous pattern $p$ is admissible if and only if $p$ is admitted by $\mathbb{Z}_+$~\cite{MariaPedro}. 
Therefore, if $p$ is admissible then there is always an ideal of a numerical semigroup that admits $p$ and contains $I$. 

An ordinary numerical semigroup is a numerical semigroup of the form $\{0,m,\rightarrow\}$. From \cite{MariaPedroAlbert}, Theorem 3.7,  we know that if $p$ is an admissible non-homogeneous linear pattern then there is an ordinary numerical semigroup that admits $p$. If $\mu$ is the smallest integer such that $\{0,\mu,\rightarrow\}$ admits $p$, then we say that $p$ is $\mu$-admissible. 
%We have $\mu= \lceil-a_0/(\sum_{i=1}^na_i-1)\rceil$ if $\sum_{i=1}^na_i> 1$ and $a_0<-2????$, otherwise $\mu=0$~\cite{MariaPedroAlbert}.%  Note the connection with pseudo-varieties~\cite{MariaPedroAlbert,RoblesRosales}. 
\begin{lemma}
The closure of an ideal $I$ of a numerical semigroup with respect to an admissible linear pattern $p$ is well-defined if $p$ is $\mu$-admissible for $\mu\leq \min(I)$. 
\end{lemma}
\begin{proof}
If $p$ is $\mu$-admissible and $\min(I)\geq \mu$, then $I\subseteq \{0,\mu,\rightarrow\}$ so there is an ideal of a numerical semigroup that contains $I$ and admits $p$, implying that the closure of $I$ with respect to $p$ is well-defined. 
\end{proof}
Note that the closure of $I$ with respect to $p$ can be well-defined although $p$ is not $\mu$-admissible for $\mu\leq \min(I)$. The smallest $m$ such that $\{0,m,\rightarrow\}$ admits the linear pattern $p(X_1)=X_1+X_2-3$ is $m=3$, so $p$ is $3$-admissible. However, the ideal $\{2,3,\rightarrow\}$ of the numerical semigroup $\mathbb{Z}_+$ also admits $p$. Therefore the closure of $I$ with respect to $p$ is well-defined for any ideal $I$ with $\min(I)\geq 2$. 
%klara: jag är inte särskilt nöjd med lemmat. 

It was proved in \cite{MariaPedro} that if $p$ is a premonic homogeneous linear pattern, then the closure of $S$ with respect to $p$ can be calculated as $$\underbrace{p(p(\cdots (p}_k(S))\cdots )),$$ denoted as $p^k(S)$, for some $k$ large enough.   
The next result generalises this to premonic non-homogeneous patterns and proper ideals of numerical semigroups. 
\begin{theorem}
\label{thm:closure}
If $I$ is an ideal of a numerical semigroup and $p(X_1,\dots,X_n)=\sum_{i=1}^na_iX_i+a_0$ is a premonic linear pattern satisfying $a_0=-(\sum_{i=1}^na_i-1)\mu$ with $\mu=\min(I)$, then $I\subseteq p(I)$ and the chain $I_0=I\subseteq I_1=p(I_0)\subseteq I_2=p(I_2)\subseteq \cdots$ stabilizes. 
The ideal $I_k=p^k(I)$ for $k$ such that $p^{k+1}(I)=p^k(I)$ is the closure of $I$ with respect to $p$. 
\end{theorem}
\begin{proof}
As in the proof of Proposition~\ref{prop:surendo}, since $p$ is premonic,  $\sum_{i=1}^ja_i=1$ for some $j\leq n$, so that $a_0=-(\sum_{i=1}^na_i -1)\mu=-\sum_{i=j+1}^na_i\mu$. 
Therefore $p(s,\dots,s,\mu,\dots,\mu)=\sum_{i=1}^ja_is+\sum_{i=j+1}^na_i\mu+a_0= s$ for all $s\in I$, implying that $I\subseteq p(I)$.    
Note also that since $p$ is admissible, by Lemma \ref{lemma:MPA}, $p(s_1,\dots,s_n)=\sum_{i=1}^{n}a_is_i+a_0 \geq \sum_{i=1}^na_i\mu+a_0= \mu$ for all non-increasing sequences $s_1,\dots,s_n\in \{\mu,\rightarrow\}$, implying that $p^k(I)\subseteq \{\mu,\rightarrow\}$ for all $k\geq 1$.  
The ideal $I$ has finite complement in $\mathbb{Z}_+$,  implying that the chain $$I_0=I\subseteq I_1=p(I_0)\subseteq I_2=p(I_2)\subseteq \cdots$$ stabilizes. 
Clearly if $p^{k+1}(I)=p^k(I)$ then $p^k(I)$ is an ideal of $S$ that admits $p$ and contains $I$. Finally, if $J$ is the closure of $I$ with respect to $p$, then $J$ must contain $p^i(I)$ for all $i\geq 1$, so that $J$ contains $p^k(I)$. Therefore $p^k(I)$ is the smallest ideal of $S$ that admits $p$ and contains $I$, so $p^k(I)$ is the closure of $I$ with respect to $p$. 
\end{proof}
Note that the conditions on the pattern in Theorem~\ref{thm:closure} are the same as the sufficient conditions for surjective endopatterns in Proposition~\ref{prop:surendo}.

\section{Giving structure to the set of patterns admitted by a numerical semigroup}
\label{sec:structures}
A numerical semigroup admits in general many patterns. 
These patterns can be combined in several ways.

\begin{lemma}
\label{lemma1}
Let $I$ be an ideal of a numerical semigroup $S$ and suppose that $p$ and $q$ are two patterns admitted by $I$. Then $p+q$ and $rp$ are also patterns admitted by $I$ for any polynomial $r$ with coefficients in $\mathbb{Z}$ such that $r(I)\geq 0$ when evaluated on any non-increasing sequence of elements from $I$. 
\end{lemma}
\begin{proof}
For all $s_1,\dots,s_n\in I$ we have $p(s_1,\dots,s_n)+q(s_1,\dots,s_n)=a+b$ for some $a,b\in S$, so that $a+b\in S$, implying that $p+q$ is a pattern admitted by $I$. 
Also, $r(s_1,\dots,s_n)p(s_1,\dots,s_n)=ab$ for some  $a\geq 0$ and $b\in S$. 
Since $ab$ is the result of adding $b$ to itself $a$ times we have that $ab\in S$, implying that $rp$ is a pattern admitted by $I$. 
\end{proof}

%\begin{lemma}
%\label{lemma1a}
%Let $I$ be an ideal of a numerical semigroup and suppose that $q_1,\dots,q_{n'}$ are endopatterns of length $n$  and $p$ is an endopatterns of length $n'$, all admitted by $I$. 
%Then, if the composition  $q\circ (p_1,\dots,p_n)$ is well-defined, then it is an endopattern of length $n$ admitted by $I$. 
%\end{lemma}
%\begin{proof}
%The image of $q_i$ is contained in $I$ which is the domain of $p$. Therefore, if $q_1(s_1,\dots,s_{n'})\geq \cdots \geq q_n(s_1,\dots,s_{n'})$ for any non-increasing sequence $(s_1,\dots,s_{n'})$ of elements in $I$, then $p\circ (q_1,\dots,q_{n})$ is a pattern of length $n$ admitted by $I$ with image contained in $I$, hence an endopattern if length $n$ admitted by $I$.
%\end{proof}
%Note that in general Lemma~\ref{lemma1a} is not true without the requirement that the pattern is an endopattern. 
%This property is important, since it makes it possible to study, for example, the effect of a power of a pattern on a numerical semigroup. 

%\begin{lemma}
%The composition of patterns is associative. 
%\end{lemma}
%\begin{proof}
%Let $p$, $q$ and $r$ be three endopatterns of an ideal $I$ of a numerical semigroup $S$. 
%\end{proof}

It can be argued that since a numerical semigroup is an additive structure, the linear patterns are the most important patterns. Note that linearity is necessary for the pattern to preserve the additivity of the numerical semigroup. 
%Some results from previous sections can be used to motivate this statement further. 

Denote by  $\mathcal{P}_n^d(I)$ the set of patterns of length at most $n$ and degree at most $d$ that are admitted by the ideal $I$ of a numerical semigroup $S$. 
Then $\mathcal{P}_n^1(I)$ is the set of linear patterns of length at most $n$ admitted by $I$. 
Lemma~\ref{lemma1}  gives algebraic structure to $\mathcal{P}_n^d(I)$. 
\begin{lemma}
\label{lemma1b}
Let $I$ be an ideal of a numerical semigroup $S$,  $n\geq 0$ and $d\geq 0$. Then $\mathcal{P}_n^d(I)$ is a semigroup with zero,  a monoid. 
\end{lemma}
\begin{proof}
By Lemma~\ref{lemma1}, if $p$ and $q$ are patterns admitted by $I$, then $p+q$ is a pattern admitted by $I$. Also, if $p$ and $q$ are of length at most $n$ and degree at most $d$, then $p+q$ is a pattern of length at most $n$ and degree at most $d$. Therefore $\mathcal{P}_n^d(I)$ is a semigroup with respect to addition. The zero pattern is admitted by any ideal of any numerical semigroup and has length at most $n$ and degree at most $d$ for any $n\geq 0$ and $d\geq 0$, so $0\in \mathcal{P}_n^d(I)$.  
\end{proof}

The set $\mathcal{P}_n^d(I)$ is not preserved by polynomial multiplication, but, if so is desired, this problem can be overcome by instead considering patterns of arbitrary degree. 
Denote by $\mathcal{P}_n(I)$ the set of patterns of length at most $n$ that are admitted by $I$.  
%It is convenient to consider the zero polynomial to be in $\mathcal{P}_n^d(S)$, although the zero polynomial has its image outside $M(S)$. 

%Then Lemma~\ref{lemma1} implies that $\mathcal{P}_n^d(S)$ is a semigroup with respect to addition. Let $R$ denote the polynomials $r\in \mathbb{Z}[X_1,\dots,X_n]$  with $r(S)\geq 0$.  
A semiring is a set $X$ together with two binary operations called addition and multiplication such that $X$ is a semigroup with both addition and multiplication, and multiplication distributes over addition.  
In general $X$ is not required to have neither zero nor unit element. 
\begin{lemma}
\label{lemma2}
Let $I$ be an ideal of  a numerical semigroup $S$. 
Then $\mathcal{P}_n(I)$ is a semiring with zero element. There is a unit element if and only if $I=\mathbb{Z}_+$.  
\end{lemma}
\begin{proof}
By Lemma~\ref{lemma1}, if $p$ and $q$ are patterns admitted by $I$, then $p+q$ and $pq$ are  patterns admitted by $I$, so $\mathcal{P}_n(I)$ is a semigroup with respect to both addition and multiplication. 
Also, clearly multiplication distributes over addition. 
Note that the pattern $p(X_1,\dots,X_n)=0$ always is a pattern admitted by $S$. 
The semiring $\mathcal{P}_n(I)$ has a unit if and only if $1\in \mathcal{P}_n(I)$, which happens if and only if $I=S=\mathbb{Z}_+$.  
\end{proof}
%Additionally, Lemma~\ref{lemma1a} can be used to make $\mathcal{P}_n(S)$ a composition semiring. 
%Note that the missing zero is due to the assumption that patterns have their image in $M(S)$. 
%There is a unit element if and only if $S=\mathbb{Z}_+$. 
%Actually, if we make the convention that $\mathcal{P}_n(S)$ contains zero, then Lemma~\ref{lemma1} says even more.
If $\Sigma$ is a (commutative) semiring with unit, then a semiring $X$ is a semi\-algebra over $\Sigma$ if there is a composition $(\sigma,x)=\sigma x$ from $\Sigma\times X$ to $X$ such that $(X,+)$ is a (left) $\Sigma$-semimodule with $(\sigma,x)=\sigma x$ and for $\sigma\in \Sigma$ and $x,y\in X$, $\sigma(xy)=(\sigma x)y=x(\sigma y)$. 
The semigroup $(X,+)$ is a (left) $\Sigma$-semimodule if $\sigma(x+y)=\sigma x+\sigma y$, $(\sigma+\rho)x=\sigma x+\rho y$, $(\sigma \rho)x=\sigma(\rho x)$ and $1\cdot x=x$ for all $\sigma,\rho\in \Sigma$ and for all $x,y\in X$.  
\begin{lemma}
Let $I$ be an ideal of  a numerical semigroup $S$ and consider the set of polynomials $R(I)=\{r\in \mathbb{Z}[X_1,\dots,X_n]: r(s_1,\dots,s_n) \geq 0 ~ \forall s_1\geq \cdots\geq s_n \in I\}$. 
Then $R(I)$ is a semiring (with zero and unit elements) and $\mathcal{P}_n(I)$ is an $R(I)$-semialgebra. 
\end{lemma}
\begin{proof}
Following the proof of Lemma~\ref{lemma2}, we see that  $R(I)$ is a  semiring with zero and unit elements.  
%In contrast with $\mathcal{P}_n(S)$, $R(S)$ does contain a zero and a unit element, so it is a monoid with respect to addition and multiplication. 
By Lemma~\ref{lemma1} we have that $rp\in \mathcal{P}_n(I)$. Also, $r(pq)=(rp)q=p(rq)$ for all $r\in R(I)$ and for all $p,q\in \mathcal{P}_n(I)$. 
Now let $r$ and $s$ be elements in $R(I)$ and $p$ and $q$ be elements in $\mathcal{P}_n(I)$. 
Then it can easily be checked that $r(p+q)=rp+rq$, $(r+s)p=rp+sp$, $(rs)p=r(sp)$ and $1\cdot p=p$, implying that $\mathcal{P}_n(I)$ is an $R(I)$-semimodule.
By Lemma~\ref{lemma2}, $\mathcal{P}_n(I)$ is also a semiring and consequently an $R(I)$-semialgebra. 
\end{proof}
%Observe that sometimes it is assumed that $A$ is a group when $A$ is an $R$-algebra over a semiring $R$. 
%This is not the case here, where $A$ is only a semigroup, and so $A$ is a semigroup-algebra over a semiring.   

%There is a set of minimal generators that generates $\mathcal{P}(S)$ as a semigroup. $These polynomials have content equal to one. 

\section{Linear patterns and a generalisation of pseudo-Frobenius}
\label{sec:pseudofrobenius}
%{\color{red} In this section patterns are still from $M(S)$ to $S$. This can be changed. But I can't decide whether zero is a pattern or not.}\medskip\\

%The constant patterns admitted by an ideal $I$ of a numerical semigroup $S$ are exactly the elements of $S$. 
%All non-constant monomials with positive coefficients are patterns for any ideal of any numerical semigroup. 
%Then Lemma~\ref{lemma1} implies that all polynomials with positive coefficients and constant term in $S$ are patterns for $I$. 
%The situation is different for monomials with negative coefficients. 
%For example, we saw in Lemma~\ref{lemmaendS} that a necessary condition for a linear pattern $p(X_1,\dots,X_n)=\sum_{i=1}^na_iX_i+a_0$ to be admissible for some numerical semigroup is that the partial sums $\sum_{i=1}^{n'}a_i\geq 0$ for $n'\leq n$ and that $a_0\in S$~\cite{MariaPedro,MariaPedroAlbert}. 
%%Indeed, if $\sigma_i<0$ for some $i$ then the polynomial $p(X_1,\dots,X_i,m,\dots,m)$ evaluated in $\overline[i]{(l,\dots,l)}$ for a large enough $l$ is smaller than 0, so it cannot belong to the numerical semigroup.
  
%We will now look at sufficient conditions for when a non-homogeneous pattern $p(X_1,\dots,X_n)=\sum_{i=1}^na_iX_i+a_0$ is induced by the homogeneous polynomial $p(X_1,\dots,X_n)-a_0$. 

Let $J$ be an ideal of a numerical semigroup $S$ and let $p$ be an endopattern of $J$. 
We will now study sufficient conditions on $a_0$ for when $p$ induces the pattern $p+a_0$ on the ideals of $S$ contained in $J$.   
We are also interested in when this implies that $p+a_0$ is an endopattern of $J$. 
Finally we will also study sufficient conditions for when the endopatterns $p_1,\dots,p_n$ induce the pattern $p_1+\cdots+p_n+a_0$ on the ideals contained in $J$. 
\begin{lemma}
If $p$ is an endopattern of $S$ and $a_0\in S$, then $p$ induces the pattern $p+a_0$ on any ideal $J$ of $S$. Additionally, $p+a_0$ is an endopattern of $S$ (but not necessarily of other ideals of $S$).  
\end{lemma}
\begin{proof}
If $p(s_1,\dots,s_n)\in S$ and $a_0\in S$, then $p(s_1,\dots,s_n)+a_0\in S$, so $p+a_0$ is admitted by $S$ and by any ideal $J$ of $S$. Endopatterns of $S$ are simply patterns admitted by $S$. 
\end{proof}
In other words, an endopattern $p$ of a numerical semigroup $S$ induces the pattern $p+a_0$ on an ideal $J$ under the condition that (i) $a_0\in S$ and (ii) the ideal $J$ is an ideal of $S$. 
It is (of course) not true that if $p$ is an endopattern of $S$ and $a_0\in S$ then $p$ induces $p+a_0$ on any ideal of any numerical semigroup. 
\begin{lemma}
If $p$ is an endopattern of $M(S)$ and $a_0\in PF(S)$, then $p$ induces $p+a_0$ on any ideal $J\subseteq M(S)$. Additionally, if $S\neq \mathbb{Z}_+$, then $p+a_0$ is an endopattern of $M(S)$ (but not necessarily of other ideals contained in $M(S)$. 
\end{lemma}
\begin{proof}
By definition of pseudo-Frobenius the monic linear pattern $f(X)=X+a_0$ is admitted by $M(S)$, implying that $f(p)=p+a_0$ is admitted by $M(S)$ and by any ideal of $S$ contained in $M(S)$.   
By Corollary~\ref{cor:monic}, since $S\neq \mathbb{Z}_+$,  $f$ is an endopattern of $M(S)$, implying that $f(p)=p+a_0$ is an endopattern of $M(S)$.  
\end{proof}
Again, this means that an endopattern $p$ of a maximal ideal $M(S)$ of a numerical semigroup $S$ induces the pattern $p+a_0$ on an ideal $J$ under the condition that (i) $a_0\in PF(S)$ and (ii) the ideal $J\subseteq M(S)$. 

Consider for example the numerical semigroup $S$ generated by $2$ and $5$.  There are no other pseudo-Frobenius than the Frobenius element, so  $PF(S)=\{3\}$. Any numerical semigroup admits the trivial pattern defined by $p(X_1)=X_1$, which is always an endopattern of the maximal ideal, and consequently the pattern $X_1+3$ is admitted by any ideal of $S$ contained in $M(S)$. Also, $X_1+3$ is an endopattern of $M(S)$. 

Note that if  $a_0\not\in S\cup PF(S)$, then an endopattern $p$ of $M(S)$ does not necessarily induce the pattern $p+a_0$ on $M(S)$. 
For example, the numerical semigroup $S=\langle 2,7\rangle=\{0,2,4,6,\rightarrow\}$ has $PF(S)=\{5\}$. We have that $M(S)$ admits the Arf endopattern $X_1+X_2-X_3$ and the non-homogeneous endopattern $X_1+X_2-X_3+5$, but $M(S)$ does not admit $X_1+X_2-X_3+3$.
However, note that $1\not\in S\cup PF(S)$ but $X_1+X_2+X_3+1$ is an endopattern of $M(S)$. 

We have seen that the pseudo-Frobenius $PF(S)$ of a numerical semigroup $S$ are related to the linear endopatterns $X_1+a_0$ of $M(S)$, with $a_0\in PF(S)$. 
By replacing the variable $X_1$ by an endopattern $p$ of $M(S)$ this resulted in a statement on for which $a_0\in \mathbb{Z}$ $p$ induces the pattern $p+a_0$. 
We will now generalise this idea in more than one direction, to sums of several patterns and to any ideal of a numerical semigroups. 
%\begin{definition}
%For $d\geq 1$, define the set $PF^d(S,I)=(S-dI)\setminus (S-(d-1)I)$ and call it \emph{the set of elements at distance $d$ from $S$ with respect to $I$}.
%\end{definition}
%Note that $PF^d(S,I)$ are the elements $a\in\mathbb{Z}$ such that for any collection of $d$ (but not for any collection of $d-1$) endopatterns $q_1(X_1,\dots,X_n),\dots,q_d(X_1,\dots,X_n)$ of $I$, the pattern $q_1(X_1,\dots,X_n)+\cdots + q_d(X_1,\dots,X_n)+a$ is also a pattern admitted by $I$ (and therefore by any ideal contained in $I$. 
%The elements at distance zero from $S$ with respect to any ideal $I$ of $S$ can be defined to be the elements in $S$, if so desired.  
%The elements at distance $n$ from $S$ with respect to $S$ is the empty set when $n\geq 1$,  reflecting the fact that the linear pattern $X_1+\cdots+X_d+a$ is admitted by $S$ if and only if $a\in S$ for all $d\geq 0$. 
%The elements at distance one from $S$ with respect to $M(S)$ is the set of pseudo-Frobenius of $S$. 

\begin{definition}
Let $I$ and $J$ be two ideals of the same numerical semigroup $S$. For $d\geq 1$, define the set $PF^d(I,J)=(I-dJ)\setminus (I-(d-1)J)$ and call it \emph{the set of elements at distance $d$ from $I$ with respect to $J$}.
\end{definition}
 
The elements at distance zero from $S$ with respect to any ideal $J$ of $S$, $PF^0(S,J)$, can be defined to be the elements in $S$, if so desired.  
The elements at distance $n$ from $S$ with respect to $S$, $PF^n(S,S)$, is the empty set when $n\geq 1$, reflecting the fact that the linear pattern $X_1+\cdots+X_d+a$ is admitted by $S$ if and only if $a\in S$ for all $d\geq 0$. 
The elements at distance one from $S$ with respect to $M(S)$, $PF^1(S,M(S)$, is the set of pseudo-Frobenius of $S$, and if $S\neq \mathbb{Z}_+$, then, by Corollary~\ref{cor:monic}, we have $PF^1(M(S),M(S))=PF^1(S,M(S))=PF(S)$. 

Note that $PF^d(S,J)$ are the elements $a\in\mathbb{Z}$ such that for any collection of $d$ (but not for any collection of $d-1$) endopatterns $q_1,\dots,q_d$ of $J$, the pattern $q_1+\cdots + q_d+a$ is also a pattern admitted by $J$, and therefore by any ideal contained in $J$. 
In general we have the following. 
\begin{lemma}
Let $I$ and $J$ be two ideals of the same numerical semigroup $S$, let $p_1,\dots,p_d$ be endopatterns of $J$ and let $a_0\in PF^d(I,J)$. Then the pattern $q=p_1+\cdots+p_n+a_0$ is admitted by any ideal $K\subseteq J$ and its image satisfies $q(K)\subseteq I$. In particular, if $I\subseteq K$, then $q$ is an endopattern of $K$. 
\end{lemma}
\begin{proof}
It is clear from the definition of $PF^d(I,J)$ that the pattern $X_1+\cdots+X_d+a_0$ is admitted by any ideal $K$ contained in $J$, and that its image is contained in $I$. 
The result follows from substituting $X_1,\dots,X_d$ with the endopatterns  $p_1,\dots,p_d$ of $J$. 
\end{proof} 

The Lipman semigroup of $S$ with respect to a proper ideal $J$ is $L(S,J)=\bigcup_{h\geq 1} (hJ-hJ)$~\cite{Lipman, Barucci}. The semigroup $L(S):=L(S,M(S))$ is called the Lipman semigroup of $S$. 
There exists a $h_0\geq 1$ such that $L(S,J)=(hJ-hJ)$ for each $h\geq h_0$, and, for each $h\geq h_0$, $(h+1)J=hJ+\mu(J)$ where $\mu(J)=\min(J)$~\cite{Barucci}. 

\begin{proposition}
When $S$ is of maximal embedding dimension, then $$PF^2(S,M(S))  = E(S)-2m(S).$$
\end{proposition}
\begin{proof}
Define $D(d,M(S))=\{z\in \mathbb{Z}: z+dM(S)\subseteq S,  z+dM(S)\not\subseteq  dM(S)\}$, that is, $D(d,M(S))=\{z\in \mathbb{Z}:z+dM(S)\subseteq S, (z+dM(S))\cap (S\setminus dM(S))\neq \emptyset\}$. 
Note that $S\setminus M(S)=\{0\}$, so that $D(1,M(S))=\{z\in \mathbb{Z}:z+M(S)\subseteq S, 0 \in z+M(S)\}$. Also, note that $S\setminus 2M(S)=E(S)$ is the set of minimal generators of $S$. 
By definition, if $d$ is such that $L(S)=(dM(S)-dM(S))$ then  $(S-dM(S))=L(S)\cup D(d,M(S))$.  

When $S$ is of maximal embedding dimension, then the Lipman semigroup of $S$ is $L(S)=(hM(S)-hM(S))$ for all $h\geq 1$, so that
$$
\begin{array}{rl}
PF^2(S,M(S))&=(S-2M(S))\setminus (S-M(S))\\
&= D(2,M(S))\setminus D(1,M(S))\\
&=\{z\in \mathbb{Z}: z+2M(S)\subseteq S,  z+2M(S)\not\subseteq 2M(S),z+M(S) \not\subseteq  S\}\\
&= E(S)-2m(S).  
\end{array}$$
\end{proof}

Compare this with the fact that the pattern $X_1+X_2-m(S)$ is admitted by a numerical semigroup if and only if $S$ is of maximal embedding dimension, and note that the smallest element in the set $E(S)-2m(S)$ is $-m(S)$. 
%\begin{lemma}
%The cardinality of $PF_{M(S)}^d(S)$ converges to  $m=m(S)$. 
%The convergence follows the convergence of the Lipman semigroup of $S$.
%\end{lemma}
%\begin{proof}
%Consider the Lipman semigroup $L(S)=\bigcup_{h\geq 1} (hM-hM)$. 
%There is a smallest  $h_0\geq 1$ such that $L(S)=hM-hM$ whenever $h\geq h_0$, in which case we also have $(h+1)M=hM+m$ (see Proposition I.2.1 in \cite{Barucci}). 
%Therefore, if $d\geq h_0$ then $z+(d+1)M=z+m+dM$ for $z\in \mathbb{Z}$ implying that  $z + (d+1)M\subseteq S$ and $z+dM\not\subseteq S$ if and only if $(z+m)+dM\subseteq S$ and $(z+m)+(d-1)M\not\subseteq S$. 
%Consequently  $(S-(d+1)M)=(S-dM)-m$ so that  $PF^{d+1}(S)=(S-(d+1)M)\setminus (S-dM)=((S-dM)-m )\setminus (S-dM)$, which has cardinality $m$. 
%\end{proof}

\begin{theorem}
The cardinality of $PF^d(I,J)$ converges to  $\mu(J)=\min(J)$. 
When $J$ is a proper ideal then the convergence follows the convergence of the Lipman semigroup of $S$ with respect to $I$.
\end{theorem}
\begin{proof}
First note that $\min(S)=0$ and $PF^d(S,J)=\emptyset$ for $d\geq 1$ and for any ideal $J$ of $S$. 

If $J$ is a proper ideal of $S$, then consider the Lipman semigroup with respect to $J$, $L(S,J)=\bigcup_{h\geq 1} (hJ-hJ)$. 
There is a smallest  $h_0\geq 1$ such that $L(S,J)=hJ-hJ$ whenever $h\geq h_0$, in which case we also have $(h+1)J=hJ+\mu(J)$ (see Proposition I.2.1 in \cite{Barucci}). 
Therefore, if $d\geq h_0$ then $z+(d+1)J=z+\mu(J)+dJ$ for $z\in \mathbb{Z}$ implying that  $z + (d+1)J\subseteq I$ and $z+dJ\not\subseteq I$ if and only if $(z+\mu(J))+dJ\subseteq I$ and $(z+\mu(J))+(d-1)J\not\subseteq I$. 
Consequently  $(I-(d+1)J)=(I-dJ)-\mu(J)$ so that  $PF^{d+1}(I,J)=(I-(d+1)J)\setminus (I-dJ)=((I-dJ)-\mu(J) )\setminus (I-dJ)$, which has cardinality $\mu(J)$. 
\end{proof}

This is not the only way to generalise the notion of pseudo-Frobenius. 
Let $S=\{0=s_0,s_1,\dots,s_n,\rightarrow\}$ be a numerical semigroup with conductor $s_n$. 
For $1\leq i \leq n$, consider the ideal $S_i=\{s\in S:s \geq s_i\}$, let $S(i)=S_i^*=(S-S_i)$ be its dual relative ideal, and let $T_i(S)=S(i)\setminus S(i-1)$.  
The \emph{type sequence} of a numerical semigroup $S$ is the finite sequence $(|T_i(S)|: 1\leq i \leq n)$~\cite{Barucci}.
Since $T_1=PF$ and $|PF|$ is the type of $S$, this is a generalisation of pseudo-Frobenius, which is different from the one in this article. 

Next we will give examples of how the sets $PF^d(I,J)$ can be used to understand small semigroups of linear patterns better. 

%\item Let $S=\mathbb{Z}_+$. Then $\mathcal{P}_1^1(S)=\{ p(X_1)=a_1X_1+a_0\in \mathbb{Z}[X_1]: a_1\geq 0, a_0> -a_1\}$. Note that if the image of the patterns was allowed to contain zero, the condition $a_0>-a_1$ would have been replaced by $a_0\in S\cup\bigcup_{i=1}^{a_1}PF^{i}(S)$. {\color{red} Jag tror att $\mathbb{Z}_+$ \"ar ett undantagsfall som inte beh\"over beaktas.} 
%\item Let $S=\left\langle 2,3\right\rangle$. Then $\mathcal{P}_1^1(S)= \{ p(X_1)=a_1X_1+a_0\in \mathbb{Z}[X_1]: a_1\geq 0, a_0\in S\cup \bigcup_{i=1}^{a_1}PF^d(S)\}$.

%\item Let $S=\mathbb{Z}_+$. Then $\mathcal{P}_1^1(S)=\{ p(X_1)=a_1X_1+a_0\in \mathbb{Z}[X_1]: a_1\geq 0, a_0\in S\cup\bigcup_{i=1}^{a_1}PF^{i}(S,M(S))\}$.  

\begin{example} 
Let $S$ be a numerical semigroup. 
Then, by definition, $M(S)$ admits the pattern $p_d(X_1,\dots,X_n)=\sum_{i=1}^dX_i+a_0$ for any $a_0\in PF^d(S,M(S))$. 
Clearly the pattern $p_d$ induces the pattern $q_d(X_1)=dX_1+a_0$. 
Therefore $\{ p(X_1)=a_1X_1+a_0\in \mathbb{Z}[X_1]: a_1\geq 0, a_0 \in S \cup \bigcup_{i=1}^{a_1}PF^i(S,M(S))\}\subseteq \mathcal{P}_1^1(M(S))$.
\end{example}

\begin{example} Let $S$ be an ordinary numerical semigroup, so that $z\in S$ for all $z\in \mathbb{Z}$ such that $z\geq m(S)$. 
Then, if $q_n(X_1)=nX_1+a_0$ is a pattern of $S$, so that $nm(S)+a_0\in M(S)$, 
we have that $p_n(s_1,\dots,s_n)=\sum_{i=1}^ns_i+a_0 \geq \sum_{i=1}^nm(S)+a_0=nm(S)+a_0$ so that $p_n(s_1,\dots,s_n)\in M(S)$, 
implying that $p_n(X_1,\dots,X_n)=\sum_{i=1}^nX_i+a_0$ is also a pattern of $S$, and so $p_n$ and $q_n$ are equivalent. 
Therefore $\mathcal{P}_1^1(S) = \{ p(X_1)=a_1X_1+a_0\in \mathbb{Z}[X_1]: a_1\geq 0, a_0\in S\cup \bigcup_{i=1}^{a_1}PF^i(S,M(S))\}$.

\end{example}
Note that if $S$ is not ordinary, then in general it is not true that  $\mathcal{P}_1^1(S) = \{ p(X_1)=a_1X_1+a_0\in \mathbb{Z}[X_1]: a_1\geq 0, a_0\in S\cup \bigcup_{i=1}^{a_1}PF^i(S,M(S))\}$. 
For example, if $S=\left\langle 3,5\right\rangle$, then $2X_1-1$ is a pattern, but $-1\in PF^3(S,M(S))$, in particular $-1\not \in PF^i(S,M(S))$ for $i\leq 2$.

\section{Numerical semigroups as the image of other numerical semigroups under linear patterns}
\label{sec:image2}
We saw in Corollary~\ref{lemmahomsem} that if $p(X_1,\dots,X_n)=\sum_{i=1}^na_iX_i$ is a homogeneous pattern admitted by the numerical semigroup $S$ then $p(S)$ is a numerical semigroup if and only if  $\gcd(a_1,\dots,a_n)=1$. However, neither Theorem~\ref{thmhomideals} nor Corollary~\ref{lemmahomsem} say anything about the numerical semigroup $p(S)$. 
Clearly, any numerical semigroup is the image of some numerical semigroup under some pattern. Indeed, any numerical semigroup is the image of itself under the pattern $p(X)=X$. 
\begin{lemma}
Any numerical semigroup $S=\left\langle a_1,\dots,a_e\right\rangle$ is the image of $\mathbb{Z}_+$ under the homogeneous pattern $p(X_1,\dots,X_e)=a_1X_1+\sum_{i=2}^{e} (a_i-a_{i-1})X_i$. 
\end{lemma}
\begin{proof}
Let $S=\left\langle a_1,\dots,a_e\right\rangle$ be a numerical semigroup, with $a_1\geq \cdots \geq a_e$ a (not necessarily minimal) set of generators of $S$. Let $p(X_1,\dots,X_e)=a_1X_1+\sum_{i=2}^{e} (a_i-a_{i-1})X_i$. Then for any non-increasing sequence $s_1,\dots,s_e\in \mathbb{Z}_+$ we have $p(s_1,\dots,s_e)= a_1s_1+\sum_{i=2}^{e}(a_i-a_{i-1})s_i= \sum_{i=1}^{e-1}a_i(s_i-s_{i+1})+a_es_e$ and since $s_i\geq s_{i+1}$ for all $i\in 1,\dots,e-1$ we have $s_i-s_{i+1}\geq 0$ so that $p(s_1,\dots,s_e)\geq 0$ and therefore $p$ is a homogeneous pattern admitted by $\mathbb{Z}_+$. Moreover, since $p(s_1,\dots,s_n)$ is of the form $\sum_{i=1}^ea_in_i$ with $n_i\geq 0$ we have that $p(s_1,\dots,s_e)\in \left\langle a_1,\dots,a_e\right\rangle=S$ so that $p(\mathbb{Z}_+)\subseteq S$. 
Now, for each generator $a_j$ of $S$, the non-increasing sequence $$s_1,\dots,s_e=\overbrace{1,\dots,1}^{j},\overbrace{0,\dots,0}^{e-j}$$ gives $p(s_1,\dots,s_n)=\sum_{i=1}^{j-1}a_i(1-1)+a_j(1-0)+\sum_{i=e-j}^{e-1}(0-0)+a_e0=a_j$, so that $a_j\in p(\mathbb{Z}_+)$.
Finally, since $p$ is linear, $p(x_1,\dots,x_e)+p(y_1,\dots,y_e)=p(x_1+y_1,\dots,x_e+y_e)$ for all sequences $x_1,\dots,x_e$ and $y_1,\dots,y_e$ in $\mathbb{Z}_+$. Note that if  $x_1,\dots,x_e$ and $y_1,\dots,y_e$ are non-increasing sequences of $\mathbb{Z}_+$ then so is $x_1+y_1,\dots,x_e+y_e$.  Therefore, for all $a,b\in p(\mathbb{Z}_+)$, also $a+b\in p(\mathbb{Z}_+)$. (Compare the proof of Lemma~\ref{lemmahomsem}.)
Consequently, we have  $S=\left\langle a_1,\dots,a_e\right\rangle \subseteq p(S)$, implying $S=p(\mathbb{Z}_+)$.
%If $\gcd(a_1,a_2-a_{1},a_e-a_{e-1})=1$, then by Lemma~\ref{?}, $p(\mathbb{Z}_+)$ is a numerical semigroup, hence closed under addition and therefore $S=\left\langle a_1,\dots,a_e\right\rangle \subseteq p(S)$, and consequently $S=p(\mathbb{Z}_+)$. 
%Finally, if $p(s_1,\dots,s_e)$ and $p(t_1,\dots,t_e)$ is in $S$, then also $p(s_1,\dots,s_e)+p(t_1,\dots,t_e)\in S$. 
\end{proof}
Note that if the numerical semigroup $S$ is the image of a numerical semigroup $S'\supseteq S$ under a pattern $p$, then $S$ admits $p$. 
Therefore it is possible to consider the chain of numerical semigroups $S_0=S\supseteq S_1=p(S_0)\supseteq S_2=p(S_1)\supseteq \cdots$. 
Observe that $p(p(S))$ is not the same as $p\circ (p,\dots,p)(S)$ (see Section~\ref{sec:composition}) and that this chain is not the same as the chain of numerical semigroups obtained in the closure of a numerical semigroup (see Definition~\ref{def:closure} and \cite{MariaPedro}). Indeed, in the closure of a numerical semigroup $S$ under a pattern $p$, the pattern is not necessarily admitted by the numerical semigroup, or, more precisely,  it is only required that $p$ is admissible (i.e. admitted by some numerical semigroup) and that $S\subseteq p(S)$. Then $p$ is admitted by $S$ if and only if $S$ is the closure of $S$ under $p$, in which case $p(S)=S$. 

Now consider for a pattern $p$ admitted by a numerical semigroup $S$ the chain $S_0=S\supseteq S_1=p(S_0)\supseteq S_2=p(S_1)\supseteq \cdots$.  The chain either stabilizes to some numerical semigroup or it does not. If it stabilizes, then it does so at once, in which case $p$ is a surjective endopattern of $S$. If it does not stabilize, then we want to explore relations between the consecutive numerical semigroups in the chain. 
The next result gives such a relation, under special conditions and when the length of the pattern is two.    
%\begin{example}
%\end{example}
%\begin{proposition}
%Let $p=\sum_{i=1}^na_iX_i$ be a linear homogeneous pattern with $\gcd(a_1,\dots,a_n)$ admitted by a numerical semigroup $S$.  
%By Lemma~\ref{?}, $p(S)$ is a numerical semigroup (contained in $S$).
%The numerical semigroup $S$ is $S=p(S)/s$ for each element $s\in S$, and also $S=p(S)/a$ with $a:=\sum_{i=1}^na_i$. 
%\end{proposition}
%\begin{proof}
%Let $s\in S$. For each element $x\in S$ is  the element $sx\in p(S)$. 

The quotient of a numerical semigroup $S$ by a positive integer $d$ is the numerical semigroup $\frac{S}{d}=\{x\in \mathbb{Z}_+:dx\in S\}$~\cite{Rosalesbok}. 
\begin{lemma}
\label{lemmaquotient}
Let $S$ be a numerical semigroup and let $p(X_1,X_2)=a_1X_1+a_2X_2$ be a linear homogeneous pattern in two variables (not necessarily admitted by $S$) such that $a_1\in S$ and $\gcd(a_1,a_2)=1$. 
Then $S=\frac{p(S)}{a_1+a_2}$. 
\end{lemma}

\begin{proof}
For all $s\in S$ we have $p(s,s)=(a_1+a_2)s$, so that $S\subseteq \frac{p(S)}{a_1+a_2}$. 
Let $x\in \frac{p(S)}{a_1+a_2}$. Then there are $s_1,s_2\in S$ such that $p(s_1,s_2)=a_1s_1+a_2s_2=(a_1+a_2)x$, implying $a_1(s_1-x)=a_2(x-s_2)$. By assumption $\gcd(a_1,a_2)=1$, and so $a_1$ must divide $x-s_2$. 
Assume that $x<s_2$, then $a_1s_1+a_2s_2=(a_1+a_2)x<(a_1+a_2)s_2<a_1s_1+a_2s_2$, but that is impossible, and therefore $x\geq s_2$ and $x-s_2\geq 0$. Now since $a_1$ divides $x-s_2$ and $a_1,s_1\in S$, it follows that $x\in S$. Therefore $\frac{p(S)}{a_1+a_2}\subseteq S$.  
\end{proof} 

It was proved in~\cite{RosalesGarcia} that every numerical semigroup is one half of infinitely many symmetric numerical semigroups. This result was extended in~\cite{Swanson} to numerical semigroups that are the quotient of infinitely many symmetric numerical semigroups by an arbitrarily integer $d\geq 2$. 
The much weaker result that every numerical semigroup is one divided by $d$ of infinitely many numerical semigroups is easy to prove, just take $dS\cup\{ds+n:s\in S\}$ for distinct positive integers $n$ with $\gcd(n,d)=1$. 
%Before proving these results it is of course necessary to know that any numerical semigroup is the quotient of infinitely many numerical semigroups (not necessarily symmetric) by an arbitrarily integer $d\geq 2$.  
However, we think that in light of Lemma~\ref{lemmaquotient}, it is interesting to see that if $S$ is a numerical semigroup, then the numerical semigroups $p(S)$ given by the linear homogeneous patterns  of length two admitted by $S$ of the form  $p(X_1,X_2)=a_1X_1+a_2X_2$,  with $a_1+a_2=d$, $a_1\in S$ and $\gcd(a_1,a_2)=1$ so that  $S=\frac{p(S)}{d}$, are all different.  
In other words, we let Lemma~\ref{lemmaquotient} imply that every numerical semigroup is the quotient of infinitely many numerical semigroups by an arbitrarily integer $d\geq 2$. 
\begin{corollary}
Let $d$ be an integer satisfying $d\geq 2$. Any numerical semigroup $S$ is the quotient from division by $d$ of infinitely many numerical semigroups of the form $p(S)$ where $p$ is a pattern of length two admitted by $S$. 
 More precisely, we have that $S=\frac{p(S)}{d}$ for all $p(X_1,X_2)=a_1X_1+a_2X_2$ such that $a_1+a_2=d$, $a_1\in S$ and $\gcd(a_1+a_2)=1$ . 
\end{corollary}
\begin{proof}
By Lemma~\ref{lemmaquotient}, any numerical semigroup $S$ is the quotient from division by $d$ of the numerical semigroup obtained as the image of $S$ by any pattern of the form $p(X_1,X_2)=a_1X_1+a_2X_2$ with $a_1+a_2=d$, $a_1\in S$ and $\gcd(a_1+a_2)=1$. 
(Note that since $a_1+a_2\geq 0$ and $\gcd(a_1,a_2)=1$, any pattern of a numerical semigroup of this form satisfies $d=a_1+a_2\geq 1$. )
There is only a finite number of pairs $(a_1,a_2)$ with $a_1,a_2>0$ and $a_1+a_2=d$, but there are infinitely many pairs $(a_1,a_2)$ with $a_1>0$, $a_2<0$, $a_1\in S$, $\gcd(a_1,a_2)=1$ and $a_1+a_2=d$. 
Let $\alpha_1$ be the smallest $a_1$ such that there is an $a_2<0$ with $a_1+a_2=d$ and let $\alpha_2=d-\alpha_1$. Then $\alpha_1=d+1$ and $\alpha_2=1$. The other pairs $(a_1,a_2)$ with $a_1>0$, $a_2<0$ and $a_1+a_2=d$ are obtained as $(a_1,a_2)=(\alpha_1+k,\alpha_2-k)$ with $k\in \mathbb{Z}_+$. 
Note that not all these pairs $(a_1,a_2)=(\alpha_1+k,\alpha_2-k)$ satisfy $\gcd(a_1,a_2)=1$. More precisely, $\gcd(a_1,a_2)=1$ if and only if $\gcd(a_1,d)=1$. Indeed, any factor of $a_1$ divides $d=a_1+a_2$ if and only if it divides $a_2$. 

Let $q_k(X_1,X_2)=(\alpha_1+k)X_1+(\alpha_2-k)X_2$. 
Clearly the set $D=\{ds: s\in S\}=\{q_k(s,s):s\in S\}\subseteq q_k(S)$ for all $k\in \mathbb{Z}_+$. 
Therefore, if $q_k(S)\neq q_{k'}(S)$, then they differ in the elements outside $D$. 
The elements in $q_k(S)\setminus D$ are of the form $q_k(s_1,s_2)$ with $s_1>s_2$, so that $s_1-s_2>0$. 
Therefore, for any $k,k'\in \mathbb{Z}_+$ with $k'>k$ we have $q_{k'}(s_1,s_2)=(\alpha_1+k')s_1+(\alpha_2-k')s_2=\alpha_1s_1+\alpha_2s_2+k'(s_1-s_2)>\alpha_1s_1+\alpha_2s_2 + k(s_1-s_2)=(\alpha_1+k)s_1+(\alpha_2-k)s_2=q_{k}(s_1,s_2)$.
%Therefore, if $t_k$ is the smallest element in $q_{k}(S)$ which is not of the form $ds$ for some $s\in S$, then $t_k=q_k(s_1,s_2)$ with $s_1>s_2$, and so $t_k>q_{k-1}(s_1,s_2)$ with $q_{k-1}(s_1,s_2)\in q_{k-1}(S)$ and since $t_k$ is not of the form $ds$ for $s\in S$, either $d\not|t_k$

Now assume that $\gcd(\alpha_1+k,d)=1$ (so that $\gcd(\alpha_1+k,\alpha_2-k)=1$ and $q_k(S)$ is a numerical semigroup). 
Let $t_{k}=(\alpha_1+k)s_1+(\alpha_2-k)s_2$ be the smallest element in $q_{k}(S)$ which is not of the form $dn$ for $n\in\mathbb{Z}_+$, that is, the smallest element in $q_{k}(S)$ which is not divisible with $d=\alpha_1+\alpha_2$. 
(Note that we proved in Lemma~\ref{lemmaquotient} that if $dn\in q_{k}(S)$, then $n\in S$ so that $dn=q_{k}(n,n)$. )
Suppose that for some $k'>k$ we have $t_{k}\in q_{k'}(S)$. 
Then there are $s'_1, s'_2\in S$ such that $t_{k}=q_{k'}(s'_1,s'_2)=(\alpha_1+k')s'_1+(\alpha_2-k')s'_2=(\alpha_1+k)s'_1+(\alpha_2-k)s'_2+(k'-k)(s'_1-s'_2)$. 
But $(\alpha_1+k)s'_1+(\alpha_2-k)s'_2\in q_{k}(S)$ and since $k'-k>0$ and $s'_1-s'_2>0$, we have $(\alpha_1+k)s'_1+(\alpha_2-k)s'_2<t_{k}$. 
Now $t_{k}$ is the smallest element in $q_{k}(S)$ not divisible by $d$, so $d$ divides $(\alpha_1+k)s'_1+(\alpha_2-k)s'_2$. 
We have $(\alpha_1+k)s'_1+(\alpha_2-k)s'_2=(\alpha_1+k)(s'_1-s'_2)+ (\alpha_1+k+\alpha_2 - k)s'_2=(\alpha_1+k)(s'_1-s'_2)+ ds'_2$. 
By assumption  $\gcd(d,\alpha_1+k)=1$,  so that $d$ divides  $s'_1-s'_2$. 
But then $d$ divides $t_{k}$, however by definition $d$ does not divide $t_{k}$, and we have a contradiction. 
Therefore $t_k\not\in q_{k'}(S)$ for $k'>k$, implying that $q_{k}(S)\neq q_{k'}(S)$ and the result follows. 
\end{proof}

\section{Composition of patterns}
\label{sec:composition}
Let $I$, $J$ and $K$ be three ideals of three numerical semigroups. Also, for $i\in 1,\dots,n'$, let  $q_i:\mathcal{S}_{n}(I)\rightarrow J$ be a pattern sending non-increasing sequences of length $n$ of elements in $I$ to elements in $J$ and let $p:\mathcal{S}_{n'}(J)\rightarrow K$ be a pattern sending non-increasing sequences of length $n'$ of elements in $J$ to elements in $K$. 
Define the polynomial composition of the patterns $p$ and $q_1,\dots,q_{n'}$  as $p\circ (q_1,\dots,q_{n'})=p(q_1(X_1,\dots,X_{n}),\dots,q_{n'}(X_1,\dots,X_{n}))$. 

Polynomial composition of patterns requires more than composition of polynomials for being well-defined. 
\begin{lemma}
\label{lemmacomposition}
The composition $p\circ (q_1,\dots,q_n')$ of the patterns $p:\mathcal{S}_{n'}(J)\rightarrow K$ and $q_1,\dots,q_n':\mathcal{S}_{n}(I)\rightarrow J$ is well-defined if the image of the $q_i$ is contained in the domain of $p$ and  $q_1(s_1,\dots,s_{n})\geq \cdots \geq q_{n'}(s_1,\dots,s_{n})$ for any non-increasing sequence $(s_1,\dots,s_{n})\in \mathcal{S}_{n}(I)$. 
\end{lemma}
\begin{proof}
Clear from the definition of pattern.
\end{proof}

\begin{example}
If $S$ is Arf, then $S$ admits $p_A(X_1,X_2,X_3)=X_1+X_2-X_3$. 
From Lemma~\ref{lemmacomposition} we know that if $p$ is a pattern of length $n$ admitted by $S$ and $q_1,\dots,q_{n}$ are patterns of length $n$ admitted by $S$ then $p\circ (q_1,\dots,q_{n})$ is also admitted by $S$, whenever that composition is well-defined.  
Therefore, since $Y_1$ and $Y_2$ are patterns, and assuming $Y_1\geq Y_2$, we can make the change of variables $X_1=X_2=Y_1$ and $X_3=Y_2$ and  we see that $q(X_1,X_2)=p_A(X_1,X_1,X_2)=2X_1-X_2$ is admitted by $S$. 
In other words, $p_A(X_1,X_2,X_3)$ induces $q(X_1,X_2)$. 
Actually, as was proved in \cite{Campillo}, it turns out that $q$ also induces $p_A$, and so $q$ and $p_A$ are equivalent. 

With other changes of variables we can obtain for example,
with $a,b,c\in\mathbb{Z}$ and $a\geq b\geq c\geq 0$,
\begin{itemize}
\item $X_1=aY_1+aY_2$
\item $X_2=bY_1+bY_2-bY_3$
\item $X_3=cY_3$
\end{itemize}
we get $p_A(X_1,X_2,X_3)=(aY_1+aY_2)+(bY_1+bY_2-bY_3)-cY_3=(a+b)Y_1+(a+b)Y_2-(b+c)Y_3)$. 
With $a=b=1$ that gives $p_A(X_1,X_2,X_3)=2p_A(Y_1,Y_2,Y_3)$ and with $a=2$, $b=c=1$ it gives $p_A(X_1,X_2,X_3)=3Y_1+3Y_2-2Y_3$. 
If instead, with $a,b,c\in\mathbb{Z}$, $a\geq b\geq c\geq 0$ we define
\begin{itemize}
\item $X_1=aY_1+aY_2-aY_3$
\item $X_2=bY_1+bY_2-bY_3$
\item $X_3=cY_3$
\end{itemize}
then we get $p_A(X_1,X_2,X_3)=(aY_1+aY_2-aY_3)+(bY_1+bY_2-bY_3)-cY_3=(a+b)Y_1+(a+b)Y_2-(a+b+c)Y_3$ which with $a=b=c=1$ gives $p_A(X_1,X_2,X_3)=2Y_1+2Y_2-3Y_3$. 
\end{example}

%\item If $S$ is Arf, then $S$ admits $p_A(X_1,X_2,X_3)=X_1+X_2-X_3$. 
%From Lemma~\ref{lemmacomposition} we know that if $p$ is a pattern of length $n$ admitted by $S$ and $q_1,\dots,q_{n}$ are patterns of length $n$ admitted by $S$ then $p\circ (q_1,\dots,q_{n})$ is also admitted by $S$, whenever that composition is well-defined.  
%Therefore, since $X_1$ and $X_2$ are patterns, and assuming $X_1\geq X_2$ we see that $q(X_1,X_2)=p_A(X_1,X_1,X_2)=2X_1-X_2$ is admitted by $S$. 
%In other words, $p_A(X_1,X_2,X_3)$ induces $q(X_1,X_2)$. 
%Also $X_1+X_2+X_3$ and $3X_3$ are patterns admitted by $S$. 
%Assuming $X_1\geq X_2\geq X_3$ we have $X_1+X_2+X_3\geq 3X_3$, so that $q(X_1+X_2+X_3,3X_3)=2(X_1+X_2+X_3)-3X_3=X_1+X_2-X_3=p_A(X_1,X_2,X_3)$. 
%Therefore  $q$ induces $p_A$, and so $p_A$ and $q$ are equivalent. (This result can be found in \cite{Campillo}.) 
%\end{itemize} 

\section*{Acknowledgements}
My gratitute goes to Ralf Fr\"oberg and Christian Gottlieb without whom this article would not have been written, and for their extensive and invaluable help during its elaboration. I also want to thank Pedro Garc\'ia-S\'anchez for pointing out that the patterns in Lemma \ref{lemma:alg}, for which the partial sums of the coefficients are larger than one, are exactly the strongly admissible patterns. 
The author acknowledges partial support from the
Spanish MEC project ICWT (TIN2012-32757) and ARES (CONSOLIDER INGENIO 2010
CSD2007-00004).

\end{document}